\documentclass[10pt,reqno]{amsart}

\usepackage{setspace}
\usepackage{amsmath}
\usepackage{amssymb}
\usepackage{amsthm}
\usepackage{float}
\usepackage{wasysym}
\usepackage{marvosym}
\usepackage{graphicx}
\usepackage{subfigure}
\usepackage{amsfonts, amstext}
\usepackage{mathtools}
\usepackage{dsfont}
\usepackage{bm}
\usepackage{hyperref}
\usepackage[top=1.5in, bottom=1.5in, left=1.5in, right=1.5in]{geometry}

\usepackage[T1]{fontenc}
\newcommand{\1}[1]{\mathbf{1}_{#1}}

\numberwithin{equation}{section}

\newtheorem{theorem}{Theorem}[section]
\newtheorem{lemma}[theorem]{Lemma}
\newtheorem{prop}[theorem]{Proposition}

\newtheorem{corollary}[theorem]{Corollary}

\theoremstyle{remark}
\newtheorem{remark}{Remark}

{\bf}{\rm}

\newcommand\norm[1]{\left\lVert#1\right\rVert}

\newcommand{\PR}{\mathbb P}

\DeclareMathOperator{\R}{\mathbb{R}}
\DeclareMathOperator{\N}{\mathbb{N}}

\makeatletter
\def\blfootnote{\xdef\@thefnmark{}\@footnotetext}
\makeatother

\begin{document}


\title[ERGM Extremal Behavior]{Extremal Behavior in Exponential Random Graphs}

\author[Ryan DeMuse]{Ryan DeMuse}

\address{Department of Mathematics, University of Denver, Denver, CO 80208,
USA} \email {ryan.demuse@du.edu}


\begin{abstract}
Yin, Rinaldo, and Fadnavis classified the extremal behavior of the edge-triangle exponential random graph model by first taking the network size to infinity, then the parameters diverging to infinity along straight lines \cite{YRF}. Lubetzky and Zhao proposed an extension to the edge-triangle model by introducing an exponent $\gamma > 0$ on the triangle homomorphism density function \cite{LZ1}. This allows non-trivial behavior in the positive limit, which is absent in the standard edge-triangle model. The present work seeks to classify the limiting behavior of this generalized edge-triangle exponential random graph model. It is shown that for $\gamma \le 1$, the limiting set of graphons come from a special class, known as Tur\'an graphons. For $\gamma > 1$, there are large regions of the parameter space where the limit is not a Tur\'an graphon, but rather has edge density between subsequent Tur\'an graphons. Furthermore, for $\gamma$ large enough, the exact edge density of the limiting set is determined in terms of a nested radical. Utilizing a result of Reiher, intuition is given for the characterization of the extremal behavior in the generalized edge-clique model \cite{Re}.

\end{abstract}

\maketitle


\section{Introduction}
The abundance of large networks in the modern era has spurred the need for robust apparatuses to study their local properties, global structure, and evolution over time \cite{F1, F2, GZFA}. Network analysis has become an increasingly important area of active research from the study of interpersonal relationships \cite{JWLWCP} and social networks \cite{RPKL} to, perhaps one of the most important and imposing networks, the human brain \cite{OD}. Exponential random graphs constitute a promising class of statistical models for network analysis. They are defined by exponential families of probability distributions over graphs with sufficient statistics characterized by functions on the graph space that capture properties of interest. These functions are commonly taken as the number of edges, two-stars, triangles, or various other subgraph structures. The flexibility of exponential random graphs allowed through their sufficient statistics means that they are adept at modeling observed networks without requiring one to prescribe a method to build the network from the ground up \cite{BSMRL, BBL}. Another aspect of exponential random graph models (ERGM) that makes them excellent candidates for modeling real-world networks is the fact that they do not assume independence between edges. While edge independence is a convenient assumption in the analysis and construction of mathematical models, it is often a hinderance when modeling many important networks. Consider the example of social networks, whose popularity has skyrocketed in recent years. Suppose there are three members of a certain social network; persons $A$, $B$, and $C$. If $A$ is a friend of $B$ and $B$ is a friend of $C$, it is more likely for $A$ to be a friend of $C$ than it would be if there were no friendships between $A$ and $B$, and $B$ and $C$. This transitivity of friendship indicates that one should adopt the assumption that edges in social networks are unlikely to appear independently.

Not only is the behavior of networks with fixed size of great interest, but also networks whose size grows over time. A natural question to ask is whether one can identify the limiting structure of a network as the number of nodes increases without bound. The current work addresses this question by characterizing the extremal behavior of the generalized edge-triangle exponential random graph model defined by Lubetzky and Zhao in \cite{LZ1}. In this generalized model, an exponent is introduced on the triangle homomorphism density function that allows the model to display non-trivial limiting behavior even in the positive limit. This extends the extremal approach of Yin, Rinaldo, and Fadnavis \cite{YRF}, also referred to as the \textit{double asymptotic framework} therein, as originated by Chatterjee and Diaconis in \cite{CD1}. Yin et al., as well as Chatterjee and Diaconis, examined the extremal behavior of the edge-triangle exponential random graph model. The edge-triangle model is defined as follows. Let $\mathcal{G}_{n}$ be the space of all labeled, finite, simple graphs on $n$ vertices. Consider the standard edge-triangle exponential random graph model on $\mathcal{G}_n$ with distribution
	\begin{equation}
		\PR_{n}^{\beta}(G_n) = \exp\left( n^2 \left(\beta_1 t(K_2, G_n) + \beta_2 t(K_3, G_n) - \psi_{n} (\beta) \right) \right),
	\end{equation}
where $K_m$ is the complete graph on $m$ vertices, $G_n \in \mathcal{G}_n$, and $\psi_{n}(\beta)$ is a normalization constant. Bhamidi, Bresler, and Sly proved that when $\beta_2 > 0$, as $n \to \infty$, a typical graph drawn from the edge-triangle ERGM looks like a $G(n,p)$ random graph, or a mixture of such graphs \cite{B}. In this region of the parameter space, the edge-triangle model does not display appreciably different behavior from the standard random graph model $G(n,p)$ \cite{EG}. Chatterjee and Diaconis studied the case of $\beta_2 \to -\infty$ while $\beta_1$ remained fixed. This approach is referred to as \textit{extremal behavior} in their seminal work because of its connections to many important theorems in extremal graph theory such as Tur\'an's theorem and Erd\H{o}s-Stone. Yin et al.\ fully characterized the extremal behavior of this model by identifying the limiting behavior as the parameters $\beta_1$ and $\beta_2$ diverge along straight lines. 

Lubetzky and Zhao \cite{LZ1} proposed an extension to the standard edge-triangle ERGM by introducing an exponent $\gamma > 0$ on the triangle density as
	\begin{equation} \label{get}
		\PR_{n}^{(\beta,\gamma)} (G_n) = \exp\left( n^2 \left(\beta_1 t(K_2, G_n) + \beta_2 t(K_3, G_n)^{\gamma} - \psi_{n} (\beta) \right) \right),
	\end{equation}
where the normalization constant $\psi_{n}^{(\beta,\gamma)}$ is given by
	\begin{equation}
		\psi_{n}^{(\beta,\gamma)} = n^{-2} \log \sum_{G_n \in \mathcal{G}_n} \exp\left(\beta_1 t(K_2, G_n) + \beta_2 t(K_3, G_n)^{\gamma}\right).
	\end{equation}
Lubetzky and Zhao investigated the regions of replica symmetry and symmetry breaking for the model defined in Equation \ref{get} for finite $\beta_1, \beta_2$. \textit{Replica symmetry} occurs when a large number of copies of $K_3$ is caused by the sheer abundance of edges present, whereas \textit{symmetry breaking} describes the region of the parameter space where a smaller number of edges appear in a particular arrangement thereby leading to a large number of copies of $K_3$. They identified an open interval $\mathcal{O}$ such that for $(\beta_1,\beta_2) \in \mathcal{O}$ and $0 < \gamma < 2/3$, symmetry breaking occurs.

The analysis provided herein seeks to characterize the asymptotic extremal properties of this generalized edge-triangle exponential random graph model for the parameters $\beta_1, \beta_2$ diverging on straight lines of the form $\beta_1 = a\beta_2 + b$ for differing constants $a$ and $b$. To identify this limit, one must turn to the space of graphons, measurable functions that form the limit object of convergent graph sequences. Following the approach taken in \cite{YRF}, the set of limiting graphons is determined. Similar to the results of \cite{YRF}, it is demonstrated that for certain values of $\gamma, a,$ and $b$, the limiting graphons come from a special class known as \textit{Tur\'an graphons}. There are also large regions of the parameter space where the limiting graphons for the model defined in Equation \ref{get} have edge density strictly between consecutive Tur\'an graphons, thereby showing that there is a region of the parameter space where the generalized model displays extremal behavior different from that seen with the standard edge-triangle model. Furthermore, for large enough values of $\gamma$, the edge density of the limiting graphon is determined as a nested radical involving the model parameters. In this pursuit, another connection to extremal graph theory that is utilized is the solution to a conjecture of Lov\'asz and Simonovits \cite{LS} concerning the asymptotic lower bound on the number of copies of $K_m$ that may appear as a subgraph in a graph on $m$ vertices with some prescribed edge density $e$. This last conjecture was settled in a sequence of papers, partially by Razborov \cite{R} and Nikiforov \cite{N}, and most recently in full generality by Reiher in \cite{Re}.

\section{Background}
For $G_n \in \mathcal{G}_{n}$ and $H$ a finite simple graph, define the homomorphism density of $H$ in $G_n$ as
	\begin{equation} \label{hom1}
		t(H,G_n) = \frac{\mathsf{hom}(H,G_n)}{n^{\left|V(H)\right|}},
	\end{equation}
where $\mathsf{hom}(H,G_n)$ is the number of edge-preserving homomorphisms of $H$ into $G_n$. A map $\varphi: V(H) \to V(G_n)$ is said to be an edge-preserving homomorphism if $(i,j) \in E(H)$ implies that $(\varphi(i),\varphi(j)) \in E(G_n)$. One can think of $t(H,G_n)$ as the probability that a map chosen uniformly at random from the set $V(G_n)^{V(H)}$ is edge-preserving. Define $T_{\beta} : \mathcal{G}_{n} \to \R$ as 
	\begin{equation}
		T_{\beta}(G_n) = \beta_1 t(K_2, G_n) + \beta_2 t(K_3, G_n)^{\gamma}.
	\end{equation}
$T_{\beta}$ is often referred to as the \textit{Hamiltonian}. Note that this can also be written as
	\begin{equation*}
		T_{\beta}(G_n) = \beta_1 \frac{2 \left|E(G_n)\right|}{n^2} + \beta_2 \left(\frac{6 \left|\Delta(G_n)\right|}{n^3}\right)^{\gamma}
	\end{equation*}
where $\Delta(G_n)$ is the set of all vertex triples in $G_n$ that form a copy of $K_{3}$. 

Some of the basics of graph limit theory are now provided. Let $\mathcal{W}$ be the space of all symmetric, measurable functions $f : [0,1]^{2} \to [0,1]$, called the \textit{graphon} or \textit{graph limit} space. From any simple graph $G$ on $n$ vertices, one can construct a corresponding graphon, $f^{G}$, such that
	\begin{equation*}
		f^{G}(x,y) = \left\{ \begin{array}{ll} 1 & \text{if }\left\{\left\lceil nx \right\rceil, \left\lceil ny \right\rceil\right\} \in E(G), \\ 0 & \text{otherwise.} \\  \end{array} \right.
	\end{equation*}
It is convenient to imagine $[0,1]$ as representing a continuum of vertices where $f(x,y) = 1$ indicates an edge being present between vertices $x$ and $y$, while $f(x,y) = 0$ corresponds to the edge being absent. For any $f \in \mathcal{W}$, define the graphon homomorphism density as
	\begin{equation} \label{hom2}
		t(H,f) = \int_{[0,1]^{\left|E(H)\right|}} \prod_{\{i,j\}\in E(H)} f(x_i,x_j) \prod_{k \in V(H)} dx_{k}.
	\end{equation}
Note that $t(H,G) = t(H,f^{G})$ for any finite simple graph $H$, so Equations \ref{hom1} and \ref{hom2} are consistent. A sequence of graphs on a growing number of vertices $\{G_{n}\}$ converges if, for every finite simple graph $H$, $t(H,G_{n})$ converges to some limit $t(H)$. Lov\'asz and Szegedy showed that for any convergent sequence $\{G_{n}\}$ there is a limit object $f \in \mathcal{W}$ such that $t(H,G_{n}) \to t(H,f)$ for every $H$ \cite{LSz}. Conversely, for any graphon $f \in \mathcal{W}$, there is a sequence of finite simple graphs on a growing number of vertices with $f$ as their limit. Define the \textit{cut norm} on the graphon space $\mathcal{W}$ such that for $f \in \mathcal{W}$,
	\begin{equation*}
		\norm{f}_{\square} = \sup_{A,B \subset [0,1]} \left| \int_{A \times B} f(x,y) \,dx \,dy \right|.
	\end{equation*}
This yields a pseudometric such that $d_{\square}(f,g) = \norm{f-g}_{\square}$ for $f,g \in \mathcal{W}$. The induced metric is given by identifying graphons $f,g$ such that $f \sim g$ if they agree on a set of full measure or there exists a measure preserving bijection $\sigma: [0,1] \to [0,1]$ such that $f(x,y) = g(\sigma x, \sigma y) = g_{\sigma}(x,y)$. A quotient space $\widetilde{\mathcal{W}}$ is obtained, referred to as the \textit{reduced graphon space}, and $d_{\square}$ becomes a metric on this space, $\delta_{\square}$, by defining 
	\begin{equation*}
		\delta_{\square}\left(\tilde{f}, \tilde{g}\right) = \inf_{\tau,\sigma} d_{\square}\left( f_{\tau}, g_{\sigma} \right).
	\end{equation*}
The reduced graphon space is also known as the space of unlabeled graphons, since the measure preserving transformation $\sigma$ corresponds to a vertex relabelling of a graphon. $(\widetilde{\mathcal{W}}, \delta_{\square})$ is a compact metric space and the homomorphism density functions are continuous with respect to the cut distance \cite{LL}. $\mathcal{W}$ may be thought of as the completion of the space of all finite simple graphs with respect to the cut distance. The advantage afforded by this is that the space $\widetilde{\mathcal{W}}$ allows one to realize all finite simple graphs, regardless of the number of vertices, as elements of the same metric space. This realization allowed for many intriguing advances, such as the aforementioned development of graph convergence and related topics, as well as the formulation of a large deviation principle for the $G(n,p)$ random graph \cite{CV} that expanded the knowledge of the behavior of exponential random graph models. For more information see \cite{LL}, \cite{BCLSV1}, and \cite{BCLSV2}.

Define $\psi_{\infty}^{(\beta,\gamma)} (T_{\beta}) = \lim_{n \to \infty} \psi_{n}^{(\beta,\gamma)} (T_{\beta})$, known as the \textit{limiting normalization constant}. Let $I: [0,1] \to \R$ be
	\begin{equation*}
		I(u) = \frac{1}{2} u \log u + \frac{1}{2} (1-u) \log(1-u).
	\end{equation*}
$I$ arises as the large deviation rate function for an i.i.d.\ sequence of Bernoulli random variables and may be extended to the space $\widetilde{\mathcal{W}}$ as
	\begin{equation*}
		I(\tilde{f}) = \int_{0}^{1} \int_{0}^{1} I(f(x,y)) \,dx \,dy.
	\end{equation*}
$I$ is well-defined and lower semi-continuous on the unlabeled graphon space. $T_{\beta}$ can be transferred from a function on the graph space to one on the unlabeled graphon space by identifying every graph with its graphon representation, then considering the equivalence class of this graphon under the previously mentioned equivalence relation. Chatterjee and Diaconis determined that the limiting normalization constant can be determined as the solution to a variational problem over the space $\widetilde{\mathcal{W}}$ and has the form
	\begin{equation} \label{limnorm}
		\psi_{\infty}^{(\beta,\gamma)}(T_{\beta}) = \sup_{\tilde{f} \in \widetilde{\mathcal{W}}} \left\{ \beta_1 t(K_2,\tilde{f}) + \beta_2 t(K_3,\tilde{f})^{\gamma} - I(\tilde{f}) \right\}.
	\end{equation}
Since $I$ is lower semicontinuous on $\widetilde{\mathcal{W}}$, $T_{\beta} - I$ is upper semicontinuous and an upper semicontinuous function on a compact space is bounded and achieves its maximum. Let $\widetilde{F}^{\ast}(\beta_1,\beta_2) = (T_{\beta} - I)^{-1} \left( \left\{ \psi_{\infty}^{(\beta,\gamma)}(T_{\beta}) \right\} \right)$, the subset of the unlabeled graphon space where the variational problem \ref{limnorm} is solved. Theorem 3.2 in \cite{CD1}, also due to Chatterjee and Diaconis, demonstrates that for a graph $G_n$ drawn according to $\PR_{n}^{(\beta,\gamma)}$, $G_n$ lies close to the set $\widetilde{F}^{\ast}(\beta_1,\beta_2)$ with exponentially high probability and
	\begin{equation} \label{setconv}
		\delta_{\square}\left( \tilde{f}^{G_{n}}, \widetilde{F}^{\ast} \right) \to 0 \hspace{.25cm} \text{in probability as} \hspace{.25cm} n \to \infty.
	\end{equation}
In fact, the rate of convergence in Theorem 3.2 of \cite{CD1} being exponential, the convergence in \ref{setconv} may be strengthed to almost sure convergence using the Borel-Cantelli Lemma. Since the parameters $\beta_1$ and $\beta_2$ are diverging along straight lines with $\beta_1 = a\beta_2 + b$, write $\widetilde{F}^{\ast}(\beta_2) = \widetilde{F}^{\ast}(\beta_1,\beta_2)$. The space of graphons allows one to identify the limiting behavior of an exponential random graph model as the network size increases. The trade-off comes from the fact that the limiting set of graphons are determined up to weak isomorphism. Two unlabeled graphons $\tilde{f}$ and $\tilde{g}$ are weakly isomorphic if $t(H,\tilde{f}) = t(H,\tilde{g})$ for all finite simple graphs $H$.

\section{Preliminaries} 

Let $G_n \in \mathcal{G}_{n}$. Define $t(G_n)$ to be the vector of graph homomorphism densities of $K_2$ and $K_3$, a single edge and the triangle graph, in $G_n$, so
	\begin{equation*}
		t(G_n) = \left( \begin{array}{c} t(K_2, G_n) \\ t(K_3, G_n) \end{array}  \right) \in [0,1]^{2}.
	\end{equation*}
Using the theory of graph limits, every finite graph has a graphon representation in the space $\mathcal{W}$ and one can consider $t$ as a function on the space $\mathcal{W}$
	\begin{equation*}
		t(f) = \left( \begin{array}{c} t(K_2,f) \\ t(K_3,f) \end{array} \right),
	\end{equation*}
where $f \in \mathcal{W}$. Define $R = \left\{ t(f): f \in \mathcal{W} \right\}$, the set of all realizable values of the edge-triangle homomorphism density vector in the graphon space. Using the convention $\N_{0} = \N \cup \{0\}$, for $k \in \N_0$, set $v_{k} = t(f^{K_{k+1}})$, where $f^{K_1}$ is the identically zero graphon, and, for $k > 1$, 
	\begin{equation}
		f^{K_{k}}(x,y) = \left\{ \begin{array}{cl} 0 & \text{if } \left\lceil xk \right\rceil = \left\lceil yk \right\rceil, \\ 1 & \text{otherwise},  \end{array} \right.
	\end{equation}
where $(x,y) \in [0,1]^{2}$. Define $f^{K_k}$ as the \textit{Tur\'an graphon with $k$ classes}. Using Equation \ref{hom2}, one finds that for all $k \in \N_0$
	\begin{equation*}
		v_{k} = \left( \begin{array}{c} \left( \frac{k}{k+1} \right) \\ \left( \frac{k(k-1)}{(k+1)^2} \right)  \end{array} \right).
	\end{equation*}
Let $f(x,y) = 1$ for all $0 \le x,y \le 1$ denote the \textit{complete graphon} and $f(x,y) = 0$ for all $0 \le x,y \le 1$ be the \textit{empty graphon}.
\begin{figure}[t!]
	\centering
	\begin{subfigure}
		\centering
		\frame{\includegraphics[keepaspectratio, width=.27\textwidth]{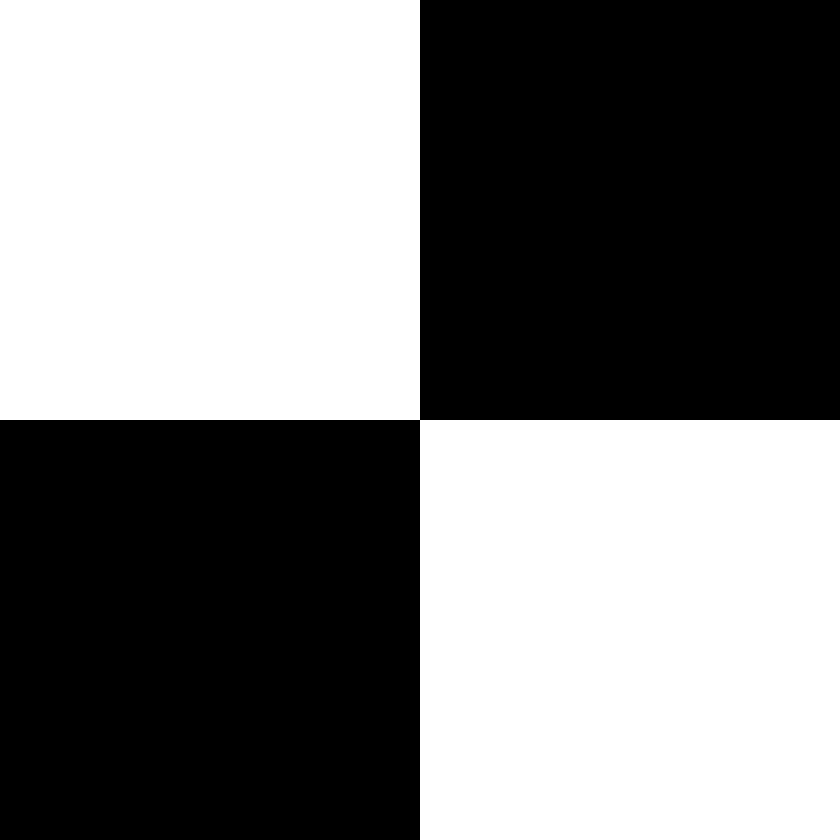}}
	\end{subfigure}
	\begin{subfigure}
		\centering
		\frame{\includegraphics[keepaspectratio, width=.27\textwidth]{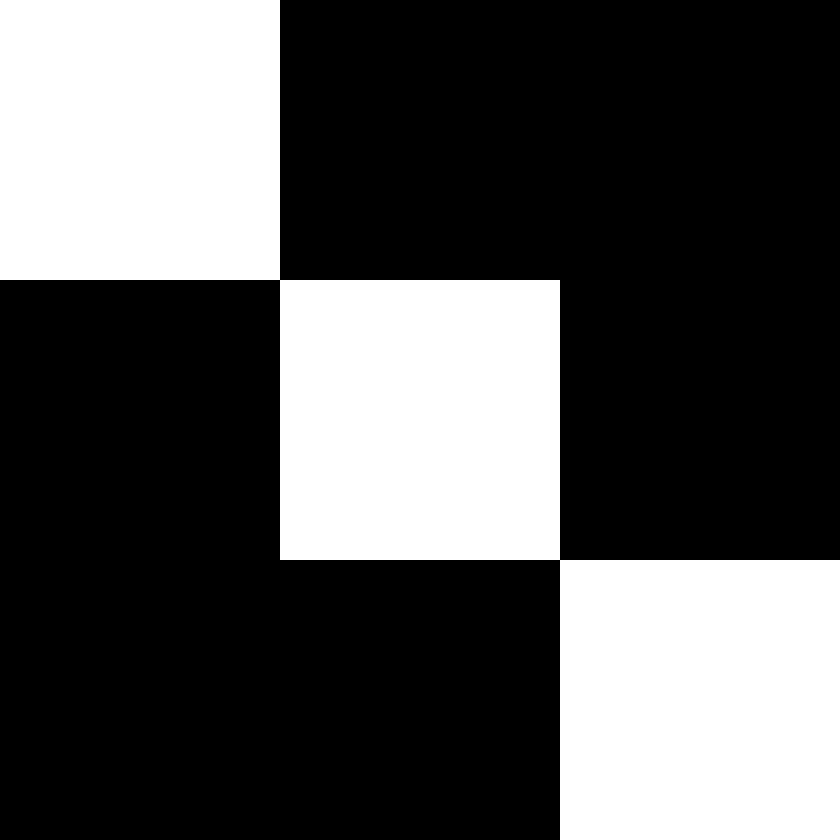}}
	\end{subfigure}
	\begin{subfigure}
		\centering
		\frame{\includegraphics[keepaspectratio, width=.27\textwidth]{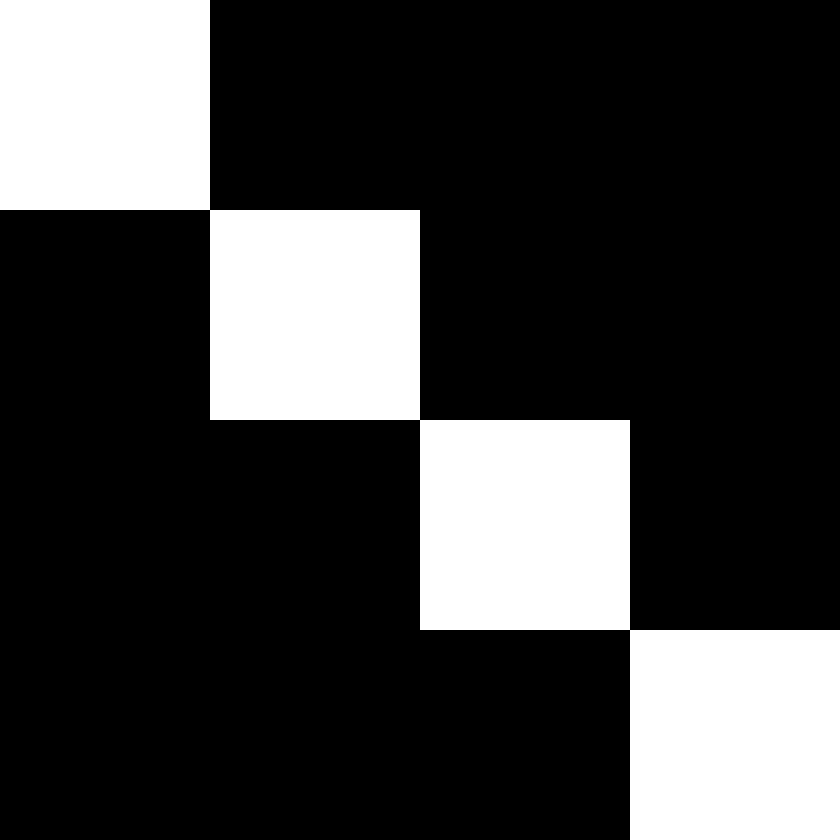}}
	\end{subfigure}
	\caption{Pixel picture representation of a T\'uran graphon with $2$, $3$, and $4$ classes respectively. Following the convention of Lov\'asz in \cite{LL}, the picture represents a function $f$ on $[0,1]^2$ where black squares have value $1$ and white have value $0$. The top left is $(0,0)$ and the bottom right is $(1,1)$.}
\end{figure}
Denote the coordinates of the vector $v_{k}$ as $e_{k}$ and $t_{k}$, so
	\begin{equation}
		e_{k} = \frac{k}{k+1} \hspace{0.3cm}\text{and}\hspace{0.3cm} t_{k} = \frac{k(k-1)}{(k+1)^2}.
	\end{equation}
The set $R$ defines the classic region of realizable edge-triangle densities. Now let $e$ denote the coordinate corresponding to the edge homomorphism density and $t$ the triangle homomorphism density. Using the Kruskal-Katona Theorem, one can derive that the upper boundary curve of $R$ is $t = e^{3 /2}$. The lower boundary of $R$ is more difficult to describe. Razborov was able to establish that, for all $k \ge 1$ and $e$ satisfying $(k-1)/k \le e \le k/(k+1)$, the triangle density $t$ is bounded below as
	\begin{equation*} \label{lb}
		t \ge \frac{ (k-1)\left( k - 2 \sqrt{k(k-e(k+1))} \right) \left( k + \sqrt{k(k-e(k+1))} \right)^2 }{k^2 (k+1)^2}
	\end{equation*}
in their seminal paper utilizing flag algebras and that this bound is tight \cite{R}. For simplicity, define $r_{k}(e)$ such that for all $k \in \N$ 
	\begin{equation} \label{raz}
		r_{k}(e) = \frac{ (k-1)\left( k - 2 \sqrt{k(k-e(k+1))} \right) \left( k + \sqrt{k(k-e(k+1))} \right)^2 }{k^2 (k+1)^2}
	\end{equation}
where $(k-1)/k \le e \le k/(k+1)$ and the endpoints of the curve $r_{k}$ are $v_{k-1}$ and $v_{k}$. Let $I_{k} = \left[\frac{k-1}{k},\frac{k}{k+1}\right]$ be the domain of $r_{k}$. Note that $r_{1}(e)$ is the constant zero function defined on the interval $I_1 = [0,1/2]$. Lemma \ref{sets} shows that the region of realizable edge-triangle densities is the closure of the homomorphism density vectors on any number of vertices.

\begin{figure}[t!]
	\centering
	\includegraphics[keepaspectratio, width=10cm]{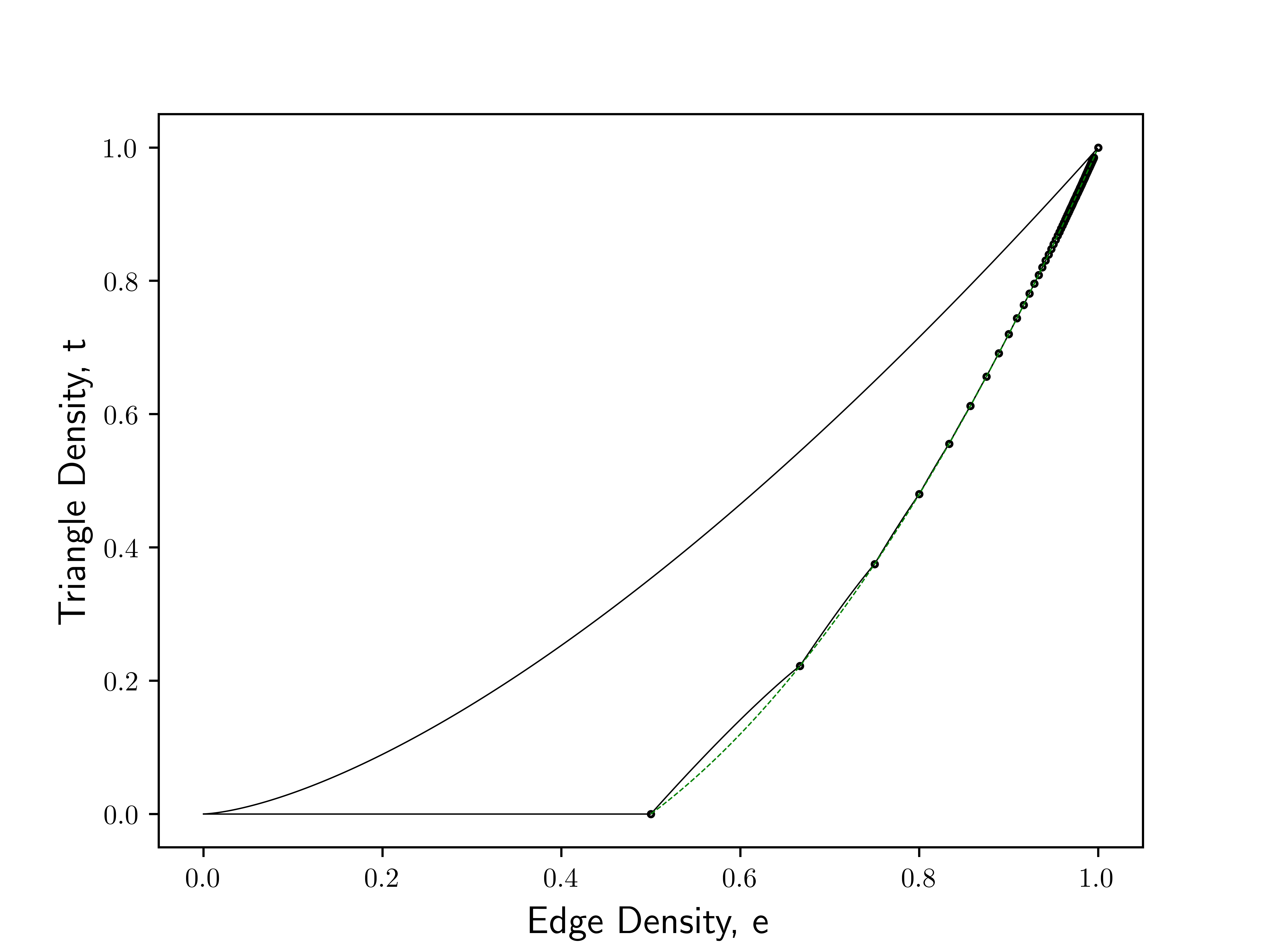}
	\caption{Graph of $R$, the region of feasible edge-triangle densities for $\gamma = 1$. The dashed line represents Goodman's bound along which the endpoints of the lower boundary segments lie. The solid line above is the one derived from Kruskal-Katona and the solid line segments below compose Razborov's bound, where the points along the lower boundary correspond to Tur\'an graphons.}
\end{figure}

\begin{lemma}[Lemma 2.1 in \cite{YRF}] \label{sets}
	Let $\mathsf{cl}(A)$ denote the topological closure of the set $A \subset \R^2$ in the usual topology. Then
		\begin{equation*}
			R = \mathsf{cl}\left( \left\{ t(G): G \in \bigcup_{n \in \N} \mathcal{G}_{n}\right\} \right).
		\end{equation*}
\end{lemma}

Let $v_{k}^{\gamma}$ be the points 
	\begin{equation*}
		v_{k}^{\gamma} = \left( \begin{array}{c} \left( \frac{k}{k+1} \right) \\ \left( \frac{k(k-1)}{(k+1)^2} \right)^{\gamma}  \end{array} \right),
	\end{equation*}
where $k \in \N_{0}$. A lower bound for the triangle density was proven by Goodman in \cite{GB}, which states in the language of homomorphism density functions that
	\begin{equation}
		t(K_{3},G) \ge t(K_2,G)\left(2t(K_2,G) - 1\right)
	\end{equation}
for any finite simple graph $G$. Goodman's bound is more crude than Razborov's, but the simplicity of Goodman's bound will prove useful for several of the results presented here. With this bound in mind, define $l(e)$ for $e \in [0,1]$ as
	\begin{equation} \label{GoodmanFunc}
		l(e) = (e(2e-1))^{\gamma}.
	\end{equation}
The points $v_{k}^{\gamma}$ lie on the graph of $l$. An important quantity in the analysis of the asymptotic structure of the probability measures $\PR_{n}^{(\beta,\gamma)}$ will be the slopes of the line segements that connect the adjacent points $v_{k-1}^{\gamma}$ and $v_{k}^{\gamma}$. Let $s_{k}(\gamma)$ be the slope of the line passing through the points $v_{k-1}^{\gamma}$ and $v_{k}^{\gamma}$, where
	\begin{equation}
		s_{k}(\gamma) = \frac{t_{k}^{\gamma} - t_{k-1}^{\gamma}}{e_{k} - e_{k-1}} = k(k+1)\left( t_{k}^{\gamma} - t_{k-1}^{\gamma} \right).
	\end{equation}
For brevity, suppress the dependence on $\gamma$ in $s_{k}(\gamma)$ so that $s_{k} = s_{k}(\gamma)$. Define vectors $\sigma_{k}(\gamma)$ as
	\begin{equation}
		\sigma_{k}(\gamma) = \left\{ \begin{array}{ll}
				(0,-1) & k = 0, \\
				\left( 1, -(s_{k}(\gamma))^{-1} \right) & k = 1,2,3,\dots \\
				\end{array}
			\right.
	\end{equation}
Similar to \cite{YRF}, these vectors will be referred to as the \textit{critical directions} of the generalized edge-triangle model. They will play an important role in determing the limiting behavior of the model. Lemma \ref{sk} characterizes the behavior of the sequence $\{s_{k}\}$ based on the values of the parameter $\gamma > 0$.

\begin{lemma} \label{sk}
The slopes of the line segments connecting adjacent points $v_{k-1}^{\gamma}$ and $v_{k}^{\gamma}$ limit to $3\gamma$ as $k \to \infty$. Furthermore, the following hold:
	\begin{itemize}
		\item For $0 < \gamma \le 5/9$, $\{s_{k}\}$ is strictly decreasing.
		\item For $5/9 < \gamma \le \log_{\frac{27}{16}}(3/2)$, $\{s_{k}\}$ is first decreasing, then transitions to increasing.
		\item For $\gamma > \log_{\frac{27}{16}}(3/2)$, $\{s_{k}\}$ is strictly increasing.
	\end{itemize} 
\end{lemma}

	\begin{proof}
		As previously mentioned, $v_{k}^{\gamma}$ lies on the graph of $l$. By the mean value theorem, for each $k \ge 2$, there exists a real number $p_{k}$ with $e_{k-1} < p_k < e_{k}$ such that $s_{k} = l^{\prime}(p_{k})$.
		Since $e_{k-1} < p_k < e_{k}$, $p_{k} \to 1$ as $k \to \infty$ and by the continuity of the derivative, $l'(p_{k}) \to l'(1)$. Therefore $s_{k} \to 3\gamma$.
		Thus the monotonicity of $\{s_k\}$ relies on the monotonicity of $l^{\prime}(e)$, and $l^{\prime \prime}$ may be written as
			\begin{equation*}
				l^{\prime \prime}(e) = \gamma (2e^2 - e)^{\gamma - 2} \left[ \left( 16(\gamma - 1) + 8 \right) e^2 - \left( 8(\gamma - 1) + 4 \right) e + (\gamma - 1) \right].
			\end{equation*}
		For $\gamma \le 1/2$, $l^{\prime}$ is strictly decreasing.
		If $\gamma > 1/2$, the quadratic piece of $l^{\prime \prime}$ can be factored with roots
			\begin{equation} \label{roots}
				x_1, x_2 = \frac{1}{4}\left( 1 \pm \frac{1}{\sqrt{2\gamma - 1}} \right)
			\end{equation}
		taking $x_1$ as the negative and $x_2$ the positive. The root $x_2 \in (1/2,1)$ when $5/9 < \gamma < 1$, which leads to a change in monotonicity for the sequence $\{s_k\}$ in this region of $\gamma$ values. If $\gamma < 5/9$, the sequence is strictly decreasing. Similarly for $\gamma \ge 1$, $\left\{ s_k \right\}$ is strictly increasing. By examining when $s_2 = s_3$ one can improve the region where the sequence is strictly increasing to $\gamma > \log_{\frac{27}{16}}(3/2)$.
	\end{proof}

For the generalized edge-triangle setting, the functions $e^{3\gamma / 2}$ and
	\begin{equation} \label{LowerBoundaryCurve}
		r(e,\gamma) = \sum_{k=1}^{\infty} r_{k}(e)^{\gamma} \1{I_k}(e)
	\end{equation}
will play important roles along the upper and lower boundary functions of the region $R$, respectively. Here $I_{k} = [e_{k-1}, e_{k}]$ and $\1{A}$ is the indicator function on the set $A$. The functions $r_{k}(e)^{\gamma}$ for a fixed $k$ are referred to as the segments of the curve $r$. $r$ is everywhere continuous on $[0,1]$ and differentiable everywhere except the points $e_k$. For $\gamma \le 1$, Lemma \ref{sk} identifies three distinct regions of behavior for the slopes $s_{k}$. In this region $r$ is concave down between the endpoints of its segments. Turning to the case where $\gamma > 1$ becomes more difficult. The segments $r_{k}^{\gamma}$ sequentially exhibit an inflection point in their domain as $\gamma$ increases. For each $\gamma > 1$, there are only finitely many lower boundary curve segments that display a change in concavity. In particular, the curves $r_{k}^{\gamma}$ such that $2 \le k \le \left\lfloor 2\left(3\gamma - 2\right) \right\rfloor$ will have a change in concavity in their respective domains and for $k > \left\lfloor 2\left(3\gamma - 2\right) \right\rfloor$, the segments $r_{k}^{\gamma}$ will be concave down.

\begin{lemma} \label{inflect}
	If $\gamma > (4+k)/6$, then the lower boundary segment $r_{k}^{\gamma}$ with domain $[e_{k-1}, e_{k}]$ changes concavity at $i_{k}$ where $i_{k} \in \left( e_{k-1}, e_{k} \right)$ and 
		\begin{equation}
			i_{k} = \frac{k}{k+1} \left( 1 - \left( \frac{1}{2(3\gamma - 2)} \right)^{2} \right).
		\end{equation}
	For $e < i_{k}$, $r_{k}^{\gamma}$ is concave up and for $e > i_{k}$, $r_{k}^{\gamma}$ is concave down. Note that for $\gamma \le 1$, $r_{k}^{\gamma}$ is strictly concave down for all $k \ge 2$.
\end{lemma}

	\begin{proof}
		Making the substitution 
			\begin{equation*}
				z = \frac{\sqrt{k(k-e(k+1))}}{k}
			\end{equation*}
		into $r_{k}(e)$ for $(k-1)/k < e < k/(k+1)$, one obtains
			\begin{equation*}
				\frac{d^{2}}{dz^{2}}\left(r_{k}^{\gamma}(z)\right) = \frac{3\gamma(k-1)^2}{2z(k+1)^2} r_{k}(z)^{\gamma - 2} \left[ (z+1)^{2} \left[ (6\gamma - 4)z - 1 \right] \right].
			\end{equation*}
		Since $z \ge 0$, there is a unique inflection point at $z = 1/(2(3\gamma-2))$, which corresponds to the point $i_{k}$ when unwrapping the substitution back to the variable $e$. The concavity properties of $r_{k}^{\gamma}$ follow. 
	\end{proof}

\section{Generalized Model \texorpdfstring{$\gamma \le 1$}{g<1}} \label{<1}

Now that some preliminaries about the important quantities involved have been established, the extremal behvior for the positive $\beta_2$ limit is investigated. A brief overview of the method to identify the extremal behavior of this model is now given. Fix $\gamma > 0$. The limiting normalization constant is determined through the variational problem defined in Equation \ref{limnorm}. Furthermore, for $G_{n}$ drawn from the distribution in Equation \ref{get}, $\delta_{\square}( \tilde{f}^{G_{n}}, \widetilde{F}^{\ast}) \to 0$ in probability as $n \to \infty$, where $\widetilde{F}^{\ast}$ is the set of graphons that maximize Equation \ref{limnorm}. As $\beta_2$ diverges, the first step is to identify whether the supremum will lie on the upper or lower boundary curve of the region $R$. If $\beta_2$ is positive, the supremum will lie on the upper curve. If $\beta_2$ is negative, the supremum becomes an infimum and the solution will lie on the lower boundary curve. As $\beta_2$ diverges, every element of $\widetilde{F}^{\ast}$ is close to the set of maximizers, $\widetilde{U}$, of Equation \ref{limnorm}.

Theorem \ref{yinrinfad} appeared in \cite{YRF} and is included for completeness. First let $n \to \infty$, then $\beta_2 \to \infty$. As $\beta_2 \to \infty$, the solution to the variational problem for the limiting normalization constant is realized along the upper boundary curve of the region $R$. This indicates that in the positive limit for $\beta_2$, the generalized model displays symmetry breaking. For a proof, see Theorem \ref{rclique} as this is a special case for $s = 3$.

\begin{theorem} [Theorem 6.1 in \cite{YRF}] \label{yinrinfad}
	Consider the generalized edge-triangle exponential random graph model defined in Equation \ref{get}. Let $\beta_1 = a\beta_2 + b$. Then
		\begin{equation}
			\lim_{\beta_2 \to \infty} \sup_{\tilde{f} \in \widetilde{F}^{\ast}(\beta_2)} \left\{ \delta_{\square} \left( \tilde{f}, \widetilde{U} \right) \right\} = 0, 
		\end{equation}
	where for $\gamma \ge 2/3$, the set $U \subset \mathcal{W}$ is:
		\begin{itemize}
			\item $U = \left\{ 0 \right\}$ if $a < -1$ or $a = -1$ and $b < 0$,
			\item $U = \left\{ 0,1 \right\}$ if $a = -1$ and $b = 0$,
			\item $U = \left\{ 1 \right\}$ if $a > -1$ or $a = -1$ and $b > 0$;
		\end{itemize}
	and for $\gamma < 2/3$, the set $U \subset \mathcal{W}$ is:
		\begin{itemize}
			\item $U = \left\{ 1 \right\}$ if $a \ge -3\gamma / 2$, and
			\item $U = \left\{ f \right\}$ if $a < -3\gamma / 2$,
		\end{itemize}
	where 
		\begin{equation}
			f(x,y) = \left\{  \begin{array}{ll} 1 & 0 \le x,y \le \left( \frac{-2a}{3\gamma} \right)^{\frac{1}{3\gamma - 2}} \\ 0 & \text{otherwise.} \end{array} \right.
		\end{equation}
\end{theorem}

Consider next the limit along horizontal and vertical lines. Along horizontal lines, $\beta_2$ is fixed and $\beta_1$ is allowed to diverge to $\infty$ or $-\infty$. In the supremum of Equation \ref{limnorm}, $(\beta_2 t(K_3,\tilde{f}) - I(\tilde{f}))$ is bounded. Thus as $\beta_1 \to \infty$, the limit is complete and, for $\beta_1 \to -\infty$, the limiting graphon is the empty graphon. With respect to vertical lines, a similar result to Theorem 7.1 of \cite{CD1} holds for the model described in Equation \ref{get}. The reason this result is included is twofold. Firstly, it is a slight generalization on the original statement of the Theorem as it applies to a larger class of models. Secondly, as was the case in \cite{YRF}, this result describes the limiting behavior of the model defined in Equation \ref{get} along the vertical directions with $a = 0$ in the specific case of $H = K_3$. The proof mimics that of Theorem 3.2 in Section 7 of \cite{YRF}. The analogous case of the parameters diverging along horizontal lines requires brief mention as well.

\begin{theorem}
	Consider the exponential random graph model
	\begin{equation} \label{get1}
		\PR_{n}^{(\beta,\gamma)}(G_n) = \exp\left( n^2 \left(\beta_1 t(K_2, G_n) + \beta_2 t(H, G_n)^{\gamma} - \psi_{n} (\beta) \right) \right)
	\end{equation}
	with $H$ an arbitrary graph different from $K_2$. Fix $\beta_1$. Let $r = \chi(H)$ be the chromatic number $H$. Let $p = e^{2\beta_1}/(1+e^{2\beta_1})$. Then
		\begin{equation}
			\lim_{\beta_2 \to -\infty} \sup_{\tilde{f} \in \tilde{F}^{\ast}(\beta_2)} \left\{ \delta_{\square}(\tilde{f},\widetilde{U}) \right\} = 0,
		\end{equation}
	where $U = \{pf^{K_{r-1}}\}$.
\end{theorem} 

	\begin{proof}
		Fix $\gamma > 0$. Let $\beta_{2}^{(i)} \to -\infty$ be an arbitrary sequence. For each $\beta_{2}^{(i)}$, let $\tilde{f}_{i}$ be an element of $\tilde{F}^{\ast}(\beta_2^{(i)})$, the set of maximizers for the variational problem. Let $\tilde{f}^{\ast}$ be a limit point of $\tilde{f}_{i}$ that exists by the compactness of $\widetilde{\mathcal{W}}$. Suppose that $t(H,f^{\ast}) > 0$. Then $t(H,f^{\ast})^{\gamma} > 0$ also. Then by the continuity of $t(H,\cdot)^{\gamma}$ and the boundedness of $t(H_1,\cdot)$ and $I(\cdot)$ on $\widetilde{\mathcal{W}}$,
			\begin{equation*}
				\lim_{i \to \infty} \psi_{\infty}^{(\beta,\gamma)}(T_{\beta_2^{(i)}}) = -\infty.
			\end{equation*}
		This contradicts the fact that for all $i$, $\psi_{\infty}^{(\beta,\gamma)}(T_{\beta_2^{(i)}})$ is bounded below by $\beta_1 - 1$, found by testing $f^{K_{r-1}}$. Thus $t(H,f^{\ast}) = 0$ and the remainder of the proof follows similarly to that of Theorem 3.2 in \cite{YRF}.
	\end{proof}

The remainder of the paper concerns $\beta_2 \to -\infty$. In this case, the supremum in Equation \ref{limnorm} will be acheived along the lower boundary curve of the edge-triangle density region. Let
	\begin{equation} \label{gfunc}
		g(e) = ae + r(e) = ae + \sum_{k=1}^{\infty} r_{k}(e)^{\gamma} \1{I_{k}}(e).
	\end{equation}
As $\beta_{2} \to -\infty$ one must minimize the function $g$ in order to solve the variational problem for the limiting normalization constant. For $\gamma < 1$, $g$ is a connected curve of concave segments and so the minimum value of $g$ can only occur at the points $e_{k}$ for $k \in \N$ or at $1$. Theorem \ref{gamma5/9} deals with the case of $\gamma \le 5/9$. Note that in any of these cases of $\beta_2 \to -\infty$, if $a \ge 0$, then the limiting graphon will be the empty graphon. For this reason, we only treat $a < 0$ throughout the case of $\beta_2 \to -\infty$.

\begin{theorem} \label{gamma5/9}
	Consider the generalized edge-triangle exponential random graph model defined in Equation \ref{get}. Let $\beta_1 = a\beta_2 + b$. Then 
		\begin{equation} \label{varproblem}
			\lim_{\beta_2 \to -\infty} \sup_{\tilde{f} \in \widetilde{F}^{\ast}(\beta_2)} \left\{ \delta_{\square} \left( \tilde{f}, \widetilde{U} \right) \right\} = 0,
		\end{equation}
	where for $\gamma \le 5/9$, the set $U \subset \mathcal{W}$ is:
		\begin{itemize}
			\item $U = \left\{ f^{K_2} \right\}$ if $a > -2$ or $a = -2$ and $b < 0$,
			\item $U = \left\{ f^{K_2}, 1 \right\}$ if $a = -2$ and $b = 0$,
			\item $U = \left\{ 1 \right\}$ if $a < -2$ or $a = -2$ and $b > 0$.
		\end{itemize}
\end{theorem}

	\begin{proof}
		Let $\beta_1 = a\beta_2 + b$ and $0 < \gamma \le \frac{5}{9}$. Firstly,
			\begin{align}
				\psi_{\infty}^{(\beta,\gamma)}(T_{\beta}) &= \sup_{\tilde{f} \in \widetilde{W}} \left\{ \beta_1 t(K_2,\tilde{f}) + \beta_2 t(K_3,\tilde{f})^{\gamma} - I(\tilde{f}) \right\} \notag \\
				&= \sup_{\tilde{f} \in \widetilde{W}} \left\{ \beta_2 \left( ae + t^{\gamma} \right) + be - I(\tilde{f}) \right\} \notag \\
				&= \beta_2 \inf_{\tilde{f} \in \widetilde{W}} \left\{ ae + t^{\gamma} + \beta_2^{-1}\left( be - I(\tilde{f})\right) \right\},
			\end{align}
		where $e = t(K_2,\tilde{f})$ and $t = t(K_3,\tilde{f})$. This preceding minimization problem must now be solved. As $\beta_2 \to -\infty$, $(be - I(\tilde{f}))$ is bounded and so $ae + t^{\gamma}$ must be minimized. Minimizing this expression occurs along the lower boundary curves of the region of realizable densities. By Lemma \ref{sk} the sequence $\{ s_{k} \}$ is strictly decreasing. Suppose that $-s_{n} \le a < -s_{n+1}$ for some $n$. So for all $k < n$, $a > -s_{k}$, and
			\begin{equation*}
				a > -s_{k} \implies a > \frac{t_{k-1}^{\gamma} - t_{k}^{\gamma}}{e_{k} - e_{k-1}} \implies a e_{k} + t_{k}^{\gamma} > a e_{k-1} + t_{k-1}^{\gamma}.
			\end{equation*}
		Thus, for all $k < n$, $a > -s_{k}$ implies that $g(e_{k-1}) < g(e_{k})$. For all $j > n$, $a < -s_{j}$ similarly implies that $g(e_{j-1}) > g(e_{j})$. Lastly, if $a = -s_{n}$, then $g(e_{n-1}) = g(e_{n})$. Now let $a < -s_{2}$. Then $a < -s_{k}$ for all $k \ge 2$. This means that $g(e_{k-1}) > g(e_{k})$ and so the minimum occurs at $1$. Suppose now that $a \ge -3\gamma$. Then $a > -s_{k}$ for all $k \ge 2$, in turn implying that $g(e_{k-1}) < g(e_{k})$ and the minimum occurs at $e_{1}$. Now the values of $g(e_1)$ and $g(1)$ must be compared where $g(e_1) = a/2$ and $g(1) = a+1$. If $a < -2$, then the minimizing value occurs at $1$. If $a > -2$, it occurs at $1/2$. Not that for all $0 < \gamma \le 5/9$, $-s_{2} < -2$. Therefore, even if $a = -s_{2}$, the minimum must occur at $1$. Comparing these values based on the value of $b$ for $a = -2$ completes the proof. 
	\end{proof}

The case of $\beta_2 \to -\infty$ along straight lines for $5/9 < \gamma \le \log_{\frac{27}{16}}(3/2)$ is an excruciating case analysis with limits being multipartite structures depending on the parameters $a, b,$ and $\gamma$. Since the segments that define $g$ are all concave down in this range of $\gamma$ values, the minimizing values for $g$ are a subset of the points $e_{k}$ and $1$. Note that the empty graph cannot be a minimizing value since for $e \in I_{1}$, $g(e) = ae$, with $a < 0$. Thus whatever the minimum value of $g$ may be, it must be less than or equal to $a/2$. 

The classfication of the model in this region of the $\gamma$ parameter space is quite technical and relies on how the values of the sequence $\{s_{k}\}$ relate to one another, which changes as a function of $\gamma$. The classification is split into $3$ technical Lemmas, which together fully characterize the model behavior in this region. Brief justification is provided as to why the result is broken into three cases. There is a critical value of $\gamma$, denoted $\gamma^{\ast}$ such that
	\begin{equation}
		\gamma^{\ast} = \frac{W_0 \left(2\ln\left(9/2\right)\right)}{\ln\left(9/2\right)},
	\end{equation}
and $s_2 = 3\gamma$ when $\gamma = \gamma^{\ast}$, where $W_0$ is the principal branch of the Lambert $W$ function. (Distinction is made between the branches $W_0$ and $W_{-1}$ of the Lambert $W$ function here because both branches will be required in Section \ref{>1}. More information about the Lambert $W$ function is provided in Section \ref{lambertw}.) For $\gamma < \gamma^{\ast}$, $3\gamma < s_{2}$ and for $\gamma < \gamma^{\ast}$, $3\gamma > s_{2}$. This produces three seperate cases for $\gamma$ that naturally seperate the potential behavior of the sequence $\{s_{k}\}$.

In the case of $\gamma \le 5/9$, the maximizing graphon was either bipartite or complete. In broad strokes, for $5/9 < \gamma \le \log_{\frac{27}{16}}(3/2)$, the maximizing graphon will be a Tur\'an graphon. Unlike for $\gamma \le 5/9$, it is possible, dependent on the parameters $\gamma$, $a$, and $b$, to realize Tur\'an graphons with any number of classes as a solution to the variational problem. Furthermore, given small, smooth, changes in the values of $a$ and $b$, there are sudden jumps in the behavior from the Tur\'an graphon on $2$ classes to Tur\'an graphons on $n$ classes for much larger values of $n$. The statement and proof of these results is left to the Appendix.

The region $\log_{\frac{27}{16}}(3/2) < \gamma \le 1$ is now analyzed. Here the sequence $\{s_{k}\}$ is strictly increasing and the Razborov curve segments are concave down leading to Tur\'an graphons in the limit.

\begin{theorem} \label{log-1}
	Consider the generalized edge-triangle exponential random graph model defined in Equation \ref{get}. Let $\beta_1 = a\beta_2 + b$. Then 
		\begin{equation}
			\lim_{\beta_2 \to -\infty} \sup_{\tilde{f} \in \widetilde{F}^{\ast}(\beta_2)} \left\{ \delta_{\square} \left( \tilde{f}, \widetilde{U} \right) \right\} = 0,
		\end{equation}
	where for $\log_{\frac{27}{16}}(3/2) < \gamma \le 1$, the set $U \subset \mathcal{W}$ is:
		\begin{itemize}
			\item $U = \left\{ f^{K_{n}} \right\}$ if $a = -s_{n}$ and $b < 0$,
			\item $U = \left\{ f^{K_{n}}, f^{K_{n+1}} \right\}$ if $a =-s_{n}$ and $b = 0$,
			\item $U = \left\{ f^{K_{n+1}} \right\}$ if $-s_{n} > a > -s_{n+1}$ or $a = -s_{n}$ and $b > 0$.
		\end{itemize}
\end{theorem}

\begin{proof}
	The proof is similar to the proof of Theorem 3.3 in \cite{YRF} and is omitted.
\end{proof}

\section{Generalized Model \texorpdfstring{$\gamma > 1$}{g>1}} \label{>1}

Recall $g$ as defined in Equation \ref{gfunc}. Note that $g$ is differentiable on the set $[0,1) \setminus \{e_{k} : k \in \N\}$ and semi-differentiable at the points $e_{k}$. Since $g$ is difficult to work with, many of the results of this section are derived by relating properties of $g$ to properties of $l$ from Equation \ref{GoodmanFunc}. Using the properties of $l$, it is determined on which segment the minimizers of $g$ must lie. Once the segment is determined, a change of variables allows one to translate the problem of finding the minimizers of $g$ to finding the root of a certain polynomial on the interval $(e_{k-1},e_{k})$, given that the minimum lies on the $k$th segment. 

Properties of $l$ must first be related to the behavior of $g$. For $k \in \N$ and $e$ in the interior of the intervals $I_{k}$, 
	\begin{equation}
		g^{\prime}(e) = a + \sum_{k=1}^{\infty} \frac{3(k-1)\gamma}{k(k+1)} \left( r_{k}(e) \right)^{\gamma - 1} \left( k + \sqrt{k(k-e(k+1))} \right) \1{I_{k}}(e).
	\end{equation} 
Other important quantites for $\gamma > 1$ are the values of the left and right derivatives of $g$ at the points $e_{k}$ where
	\begin{equation}
		\partial_{-} g \left( e_{k} \right) = a + \frac{3(k-1)\gamma}{k+1} t_{k}^{\gamma - 1},
	\end{equation}
and
	\begin{equation}
		\partial_{+} g \left( e_{k} \right) = a + \frac{3k\gamma}{k+1} t_{k}^{\gamma - 1}.
	\end{equation}
These derivatives and the properties of the Goodman bound will yield results about the extremal behavior of the generalized model for $\gamma > 1$. Lemma \ref{inflect} showed that the segments of $g$ sequentially display inflection points as $\gamma$ increases from $1$. The variational problem amounts to finding the minimum of $g$ on its domain. Since $g$ is continuous on a closed interval, the minimum of $g$ is attained on its domain. It will be helpful to understand how the concavity changes of $g$ affect where the minimum can occur. To this end, the Goodman bound is employed to help restrict the search. Since $r_{k}(e)$ is increasing and concave for all $k$ and $e$ in their respective domains and $j(e) = (2e^2 - e)$ is convex, it must be true that $r_{k}(e) \ge j(e)$ for all $k \in \N$ and $e \in \left[e_{k-1},e_{k}\right]$ with equality only when $e = e_{k-1}$ or $e = e_{k}$. Thus $r_{k}(e)^{\gamma} \ge l(e)$. Furthermore, this implies that $r(e)^{\gamma} \ge l(e)$ for all $e \in \left[0,1\right]$. Recall that if $f$ is strictly on convex on $[a,b]$ and $a \le x < y < z \le b$, then
	\begin{equation} \label{convex}
		\frac{f(z) - f(x)}{z - x} < \frac{f(z) - f(y)}{z - y}.
	\end{equation}
Consider the slope of the secant line through the points $e_k$ and $e$ on the function $l$ where $e \le e_{k-1}$. Since $l$ is strictly convex for $\gamma \ge 1$, the slope of the secant line increases as $e$ increases to $e_{k-1}$ and is bounded above by $s_{k}$. Furthermore, $r(e)^{\gamma}$ lies above $l$, so the slope of the secant line on the graph of $r^{\gamma}$ between the points $e_k$ and $e \le e_{k-1}$ is positive and bounded above by $s_{k}$. Thus for all $e \le e_{k-1}$
 	\begin{equation*}
 		\frac{t_{k}^{\gamma} - r(e)^{\gamma}}{e_{k} - e} \le \frac{t_{k}^{\gamma} - l(e)^{\gamma}}{e_{k} - e} \le \frac{t_{k}^{\gamma} - t_{k-1}^{\gamma}}{e_{k} - e_{k-1}} = s_{k} < -a.
 	\end{equation*}
So for $-s_{k} > a$ and $e \le e_{k-1}$, it holds that $g(e_k) < g(e)$. Similar reasoning can be applied for $a \ge -s_{k+1}$ and $e \ge e_{k+1}$ to show that $g(e) \ge g(e_{k})$, where equality holds only when $a = -s_{k+1}$ and $e = e_{k+1}$.
This leads to the conclusion that for $\gamma > 1$ and $-s_{k} > a \ge -s_{k+1}$, the minimizing value for $g$, $e^{\ast}$, is contained in the interval $(e_{k-1},e_{k+1}]$. This information limits the search for the minimizing value of $g$ when $\gamma > 1$. The variational problem is first solved for $-a = s_{k}$ and $k \ge 2$, then later the case where $-s_{k} > a > -s_{k+1}$. The case for $\gamma > 1$ and $a = -s_{2}$ is studied first because of its relative simplicity with respect to $a = -s_{k}$ for $k \ge 3$.

\begin{figure}[t!]
	\centering
	\frame{\includegraphics[keepaspectratio, width=5cm]{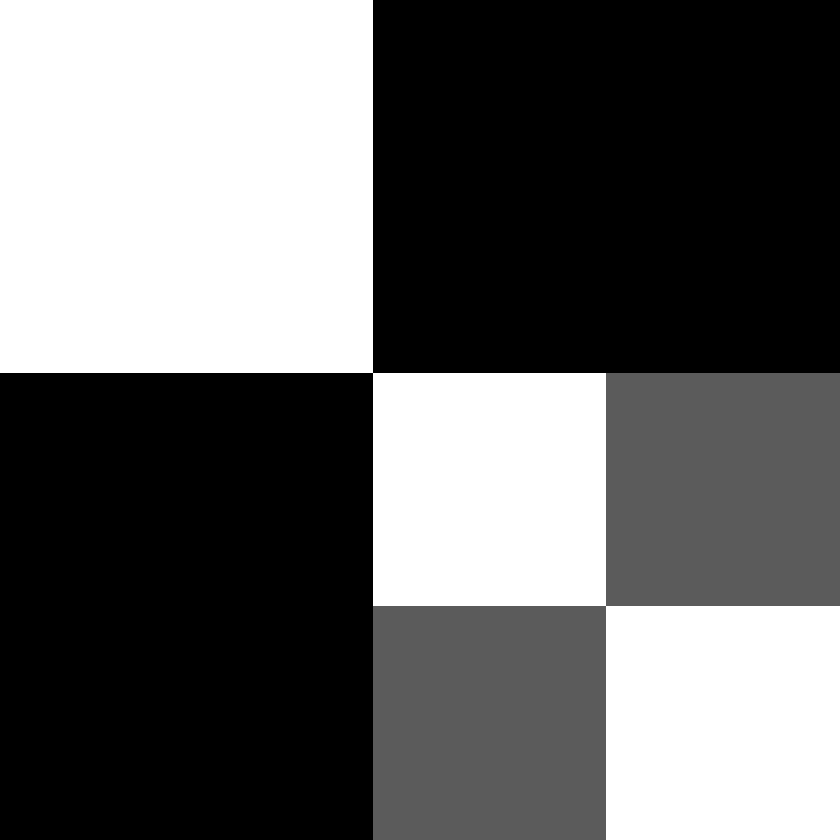}}
	\caption{Representation of the limiting graphon, $f$, for $\beta_2 \to -\infty$ and $a = -s_{2}$ in Proposition \ref{s2}. The black squares have value $1$, white squares value $0$, and gray is some $0 < p < 1$ chosen such that the $t(K_2,f) = e^{\ast}$. Note that $p$ is close to $0$ for $e^{\ast}$ near $1/2$ and $p$ is close to $1$ for $e^{\ast}$ near $2/3$. For a precise definition of $f$, see \cite{RS}.}
	\label{picLS}
\end{figure}

\begin{prop} \label{s2}
	If $a = -s_2$ and $\gamma > 1$, then there exists a unique $e^{\ast}$ with $1/2 < e^{\ast} < i_2$ such that $g$ has a global minimum at $e^{\ast}$.
\end{prop}

\begin{proof}
	Let $a = -s_2$. By the above, the minimum of $g$ is some $e^{\ast} \in (e_1,e_3]$. Since $\{s_{k}\}$ is increasing, $a > -s_{k}$ for all $k > 2$. Therefore $g(e_{k}) > g(e_{k-1})$ for all $k > 2$. Recall
		\begin{equation*}
			\partial_{+} g (e_{k-1}) = a + \frac{3(k-1)\gamma}{k} t_{k-1}^{\gamma - 1}.
		\end{equation*}
	It suffices to show that $\partial_{+} g(e_{k-1}) \ge 0$ for $k \ge 2$ and that $\partial_{+} g(e_{1}) < 0$. Note that $\partial_{+} g(e_{k-1})$ is an increasing function of $k$ and
		\begin{equation*}
			\partial_{+} g (e_{2}) = (9\gamma - 6)\left( \frac{2}{9} \right)^{\gamma} > 0.
		\end{equation*}
	Thus $\partial_{+} g (e_{k-1}) \ge 0$ for $k \ge 2$, and $\partial_{+}g (e_1) = -s_2 < 0$. The concavity properites of $g$ ensure the uniqueness of $e^{\ast}$.
\end{proof} 
This Proposition implies that for $\beta_1 = a\beta_2 + b$, $\gamma > 1$, and $a = -s_2$,
	\begin{equation*}
		\lim_{\beta_2 \to -\infty} \sup_{\tilde{f} \in \widetilde{\mathcal{W}}} \left\{ \delta_{\square} (\tilde{f}, \tilde{U}) \right\} \to 0
	\end{equation*}
where, for $\gamma > 1$ and $a = -s_{2}$, the set $U = \{f\}$ with $t(K_2,f) = e^{\ast}$ and $t(K_3,f) = r_2(e^{\ast})$ for $1/2 < e^{\ast} < i_2$. The edge density is obtained through the variational problem; and, since $\beta_2 \to -\infty$, the triangle density lies on the Razborov curve. Thus the limiting set of graphons is the set of all graphons with edge density $e^{\ast}$ and triangle density $r_2(e^{\ast})$, found by substituting the edge density into the Razborov curve segment. Radin and Sadun determined a precise formulation of the graphons that lie on the Razborov curve. Figure \ref{picLS} illustrates an example of what this graphon looks like. For a precise definition of this graphon construction, see Theorem 4.2 in Section 4 of \cite{RS}.

With this first case established, consider now the situation that arises when $a = -s_{n}$ for $n > 2$. The simplicity of this first case stemmed from the fact that $t_1 = 0$. Since $t_{n-1} > 0$ for $n > 2$, the subsequent cases present a greater challenge for finding the minimizing value of $g$. The Lambert $W$ function will be required and some of its basic properties are discussed in the next section. This function is then used to procure a sequence of critical $\gamma$ values that determine whether the minimum of $g$ occurs at one of the endpoints $e_n$ or on the interior of the interval $I_{n}$.

\subsection{Lambert W Function} \label{lambertw}

Let $f(z) = ze^z$ for $z \in \mathbb{C}$. The \textit{Lambert W function} is a set of functions, namely the branches of the inverse relation to the function $f(z)$. These branches satisfy the functional equation
	\begin{equation}
		W_i (z) e^{W_i (z)} = z
	\end{equation}
for all $i \in \mathbb{Z}$. This set contains two real valued branches, $W_{-1}$ and $W_0$, such that
	\begin{equation}
		W_{-1} : [-e^{-1},0) \to (-\infty,-1] \hspace{.5cm}\text{and}\hspace{.5cm} W_{0} : [-e^{-1},\infty) \to [-1,\infty).
	\end{equation}
The branch $W_{-1}$ is decreasing on its domain from $-1$ to $-\infty$, and the branch $W_0$ is increasing on its domain from $-1$ to $\infty$. For more information on the Lambert $W$ function, see \cite{CGHJK}.

The central piece to the proof of Proposition \ref{s2} was the fact that for all $\gamma > 1$, $\partial_{+} g(e_1) < 0$. This is true because
	\begin{equation*}
		\partial_{+} g(e_1) = -s_{2} + \frac{3\gamma}{2} t_{1}^{\gamma - 1} = -s_{2}.
	\end{equation*}
For arbitrary $n \ge 3$, this is indeed not nearly as simple. If $a = -s_{n}$, it is required that
	\begin{equation} \label{partialright}
		\partial_{+} g(e_{n-1}) = -s_{n} + \frac{3(n-1)\gamma}{n} t_{n-1}^{\gamma - 1} < 0.
	\end{equation}
Note that $t_{n-1} > 0$ for $n \ge 3$. If Inequality \ref{partialright} holds, then there is a local minimum of $g$ at some $e^{\ast} \in (e_{n-1},i_{n})$ and, if the inequality does not hold, the minimum of $g$ is at one, or both, of the endpoints $e_{n-1}$ and $e_{n}$. Utilizing the defintion of $s_n$ and $t_n$, the inequality
	\begin{equation} \label{W0}
		\frac{3(n-1)\gamma}{n} t_{n-1}^{\gamma - 1} < s_{n}
	\end{equation}
is equivalent to
	\begin{equation} \label{W1}
		1 + \frac{3\gamma}{(n+1)(n-2)} < \left( \frac{n^3}{(n+1)^2 (n-2)} \right)^{\gamma}.
	\end{equation}
The Inequality \ref{W1} is of the form
	\begin{equation*}
		1 + a(n)\gamma < (p(n))^{\gamma}
	\end{equation*}
where
	\begin{equation}
		a(n) = \frac{3}{(n+1)(n-2)} \hspace{0.25cm}\text{and}\hspace{0.25cm} p(n) = \frac{n^3}{(n+1)^2 (n-2)}.
	\end{equation}
Let $a = a(n)$ and $p = p(n)$. The corresponding Equality to Inequality \ref{W1} must be solved for $\gamma$. This reveals a range of $\gamma$ values such that for $n \ge 3$, the desired inequality holds and a local minimum of $g$ is guaranteed on the interior of $I_{n}$. Using a simple substitution of variables, solutions to equations of the form
	\begin{equation} \label{W2}
		1 + a\gamma_{n} = p^{\gamma_{n}}
	\end{equation}
can be determined using the Lambert W function, where $\gamma_{n}$ is the solution to Equation \ref{W2} for a fixed $n \in \N$. Note that for $f(x) = x e^{x}$, $f$ is not injective on $(-\infty,0)$ and the image of $f$ on this interval is $[-e^{-1},0)$. Thus for $x < 0$ and $x \neq - 1$, the equation $x e^x = y e^y$ has two solutions in $y$. The first is apparent upon inspection as $y = x$, while the other is given by $y = W_{-1}(xe^x)$ for $x \in (-1,0)$. Letting $-t = \gamma_{n} + 1/a$, Equation \ref{W2} becomes
	\begin{equation} \label{W3}
		t\ln{p} e^{t\ln{p}} = -\frac{1}{a}\ln{p}e^{-\frac{1}{a}\ln{p}}.
	\end{equation}
One solution of this equation is $t = -1/a$. This turns out to be the solution along the branch $W_{0}$ corresponding to $\gamma = 0$. This can also be determined directly from Equation \ref{W2}. There is another solution to this equation along the branch $W_{-1}$. Let
	\begin{equation}
		q(n) := -\frac{1}{a(n)}\ln{p(n)} = -\frac{(n+1)(n-2)}{3} \ln\left( \frac{n^3}{(n+1)^2 (n-2)} \right).
	\end{equation}
The function $q(n)$ is a negative and decreasing for $n \ge 3$. Taking the limit as $n \to \infty$, $q(n) \to -1$. Since $q(n) \in (-1,0)$, $q(n) e^{q(n)} \in (-e^{-1},0)$ for all $n \ge 3$. This implies that there is a second solution to the Equation \ref{W2} along the branch $W_{-1}$. Applying $W_{-1}$ to both sides of Equation \ref{W3}, one obtains the nonzero solution
	\begin{equation} \label{gamman}
		\gamma_{n} = -\left( \frac{1}{\ln{p(n)}} W_{-1}\left( q(n) e^{q(n)} \right) + \frac{1}{a(n)} \right).
	\end{equation}
Since $W_{-1}$ is monotone decreasing on the interval $(-e^{-1},0)$, if $\gamma > \gamma_{n}$, then the Inequality \ref{W0} is satisfied. Unfortunately, the preceding equation is not insightful as to how fast $\gamma$ must grow as a function of $n$. To this end, a recent result of Chatzigeorgiou is employed. In 2013, Chatzigeorgiou determined upper and lower bounds on the Lambert function $W_{-1}$ \cite{C}.

\begin{theorem}[Theorem 1 in \cite{C}] \label{LambertBound}
	The Lambert function $W_{-1}(-e^{-u-1})$ for $u > 0$ is bounded as
		\begin{equation}
			-1 - \sqrt{2u} - u < W_{-1}(-e^{-u-1}) < -1 - \sqrt{2u} - \frac{2}{3} u.
		\end{equation}
\end{theorem}
Rewrite the expression inside $W_{-1}$ from Equation \ref{gamman} as
	\begin{equation}
		q(n) e^{q(n)} = -e^{-\left( -q(n) - \ln\left( -q(n) \right) - 1 \right) - 1}
	\end{equation}
and let $u(n)$ be defined as
	\begin{equation} \label{un}
		u(n) = -q(n) - \ln\left( -q(n) \right) - 1.
	\end{equation}
Since $xe^x$ is increasing for $x \ge -1$ and $q(n) > -1$, $q(n)e^{q(n)} > -e^{-1}$. Using this and the fact that $-e^{-x-1}$ is increasing, $-e^{-u(n) - 1} > -e^{-1}$ implies that $u(n) > 0$. Thus Theorem \ref{LambertBound} can be applied to $u(n)$. Using Theorem \ref{LambertBound}, Lemma \ref{gammabound} determines an asymptotic equivalence formula for the solutions $\gamma_{n}$ as a linear function of $n$.

\begin{lemma} \label{gammabound}
	Let $\gamma_{n}$ be the sequence defined in Equation \ref{gamman}, that is, $\gamma_{n}$ is the sequence of solutions to Equations \ref{W2} for all $n \in \N$. Then $\gamma_{n} \sim \frac{2n}{9}$.
\end{lemma}

	\begin{proof}
		Let $W_{n} = -W_{-1}(q(n)e^{q(n)})$. Recall that $q(n) \in (-1,0)$ for all $n$ and $q(n) \to -1$. This further implies that $W_{n} \to 1$ as $n \to \infty$. Next, it is shown that $n(W_{n}-1) \to 1/3$. Using simple bounds on $\log$, $q(n)+1$ can be bounded as
			\begin{equation*}
				\frac{1}{3(n+1)} \le q(n)+1 \le \frac{1}{3(n+1)} + \frac{(3n+2)^{2}}{6(n+1)^3 (n-2)}.
			\end{equation*}
		Thus $q(n) + 1 \sim 1/(3n)$ and
		$u(n) \sim 1/(18n^2)$ where $u(n)$ is defined in Equation \ref{un}. Combining this asymptotic equivalence and the bound in Theorem \ref{LambertBound}, it follows that $n\left(W_{n} - 1\right) \to 1/3$. Denote the Laurent series of $\ln(1+z)^{-1}$ about $z=0$ as
			\begin{equation}
				\frac{1}{\ln(1+z)} = \sum_{k=-1}^{\infty} c_{k} z^{k} = \frac{1}{z} + \frac{1}{2} - \frac{z}{12} + \frac{z^2}{24} - \cdots .
			\end{equation}
		The coefficients $c_{k}$ satisfy the recurrence
			\begin{equation*}
				c_{k} = \sum_{i=2}^{k+2} \frac{(-1)^{i} c_{k+1-i}}{i},
			\end{equation*}
		and induction shows that $\left|c_{k}\right| \le 2^{k}$ for all $k \ge 0$. 
		Thus
			\begin{align} \label{gammalimit}
				\lim_{n \to \infty} \frac{9\gamma_{n}}{2n} &= \lim_{n \to \infty} \frac{9}{2n} \left(\frac{-W_{-1}\left(q(n)e^{q(n)}\right)}{\ln(p(n))} - \frac{1}{a(n)}\right) \notag \\
					&= \lim_{n \to \infty} \frac{3(n+1)(n-2)}{2n(3n+2)} \left[ 3(W_{n}-1)n + (3W_{n}-2) \right] + \frac{9W_{n}}{4n} \notag \\
					&\hspace{1cm}+ \frac{9W_{n}}{2n} \sum_{k=1}^{\infty} c_{k} \left( \frac{3n+2}{(n+1)^2(n-2)} \right)^{k} \notag \\ 
					&= 1.
			\end{align}
		The first term in the preceding limits to $1$ as $n$ tends to infinity and $(9W_{n})/(4n) \to 0$. Lastly, the infinite sum limits to $0$ by the dominated convergence theorem.  
	\end{proof}

\subsection{Solution of Variational Problem}

Previously examined was the special case of the solution when $a = -s_{2}$, which gave some insight into the general case for $\gamma > 1$. The next two lemmas concern the extreme cases of parameter values $a$ and $b$. The extreme cases for this regime of $\gamma$ values is when either $a > -s_{2}$ or $a \le -3\gamma$. The variational solution to these cases is presented at the beginning of this section as they require only brief justification.

\begin{lemma} \label{k2}
	Consider the generalized edge-triangle exponential random graph model defined in Equation \ref{get}. Let $\beta_1 = a\beta_2 + b$. Then Equation \ref{varproblem} holds, and, for $\gamma > 1$ and $0 > a > -s_2$, the set of maximizers $U \subset \mathcal{W}$ is $U = \{f^{K_2}\}$.
\end{lemma}

	\begin{proof}
		Assuming $0 > a > -s_{2}$, $\partial_{+} g(e_1) > 0$. 
		Note that $g^{\prime}(e) = a < 0$ by assumption, since $g(e) = ae$ for $e \le e_1$. Thus $g$ is decreasing for $e < e_{1}$ and increasing for $e > e_{1}$, and the minimum of $g$ must occur at $e_{1}$. 
	\end{proof}

\begin{lemma} \label{1}
	Consider the generalized edge-triangle exponential random graph model defined in Equation \ref{get}. Let $\beta_1 = a\beta_2 + b$. Then Equation \ref{varproblem} holds, and, for $\gamma > 1$ and $a \le -3\gamma$, the set of maximizers $U \subset \mathcal{W}$ is $U = \{1\}$.
\end{lemma}

	\begin{proof}
		Note that for $e \neq e_{k}$ for any $k \ge 2$ and $e \in (e_2,1)$,
			\begin{equation*}
				g^{\prime}(e) = a + 3\gamma \sum_{k =1}^{\infty} \frac{(k-1)(k+\sqrt{k(k-e(k+1))})}{k(k+1)} \left( r_{k}(e) \right)^{\gamma - 1} \1{I_{k}}(e).
			\end{equation*}
		For $e \in (e_{k-1},e_{k})$, 
			\begin{equation*}
				k < k+\sqrt{k(k-e(k+1))} < k+1,
			\end{equation*}
		which implies that for all $k \ge 2$ and $e \in (e_{k-1},e_{k})$
			\begin{equation*}
				0 < \frac{(k-1)(k+\sqrt{k(k-e(k+1))})}{k(k+1)} \left( r_{k}(e) \right)^{\gamma - 1} < 1.
			\end{equation*}
		Therefore $g^{\prime}(e) < 0$ and the minimizing value for $g$ on its domain is $1$.
	\end{proof}

The remainder of the case $\gamma > 1$ is split into two seperate theorems. By Lemma \ref{sk} the sequence $\{s_{k}\}$ is increasing. The first theorem pertains to $a = -s_{n}$ for some $n \in \N$ and the second regards $-s_{n} > a > -s_{n+1}$ for some $n \in \N$. This requires the critical sequence of $\gamma$ values, $\{\gamma_{n}\}$ as defined in Section \ref{lambertw}, such that for $\gamma > \gamma_{n}$ and certain values of $a$ the solution to the supremum problem is not a Tur\'an graphon, but rather some graphon with edge density between two subsequent Tur\'an graphons.

\begin{theorem} \label{equal}
	Consider the generalized edge-triangle exponential random graph model defined in Equation \ref{get}. Let $\beta_1 = a\beta_2 + b$. Then Equation \ref{varproblem} holds; and, for $\gamma > 1$ and $a = -s_{n+1}$, the set of maximizers $U \subset \mathcal{W}$ is
		\begin{itemize}
			\item $U = \{ f^{K_{n+1}} \}$ if $b < 0$ and $\gamma \le \gamma_{n+1}$,

			\item $U = \{ f^{K_{n+1}}, f^{K_{n+2}} \}$ if $b = 0$ and $\gamma \le \gamma_{n+1}$,

			\item $U = \{ f^{K_{n+2}} \}$ if $b > 0$ and $\gamma \le \gamma_{n+1}$, and

			\item $U=\{ f \}$ if $\gamma > \gamma_{n+1}$ where $t(K_2,f) = e^{\ast}$ and $t(K_3,f) = r_{n+1}(e^{\ast})$ for some $e^{\ast} \in (e_{n},i_{n+1})$ with $r_{n+1}$ defined in Equation \ref{raz}.
		\end{itemize}
\end{theorem}

	\begin{proof}
		Let $\beta_1 = a\beta_2 + b$ and $\gamma > 1$. Then 
			\begin{align}
				\psi_{\infty}^{(\beta,\gamma)}(T_{\beta}) &= \sup_{\tilde{f} \in \widetilde{W}} \left\{ \beta_1 t(K_2,\tilde{f}) + \beta_2 t(K_3,\tilde{f})^{\gamma} - I(\tilde{f}) \right\} \notag \\
				&= \sup_{\tilde{f} \in \widetilde{W}} \left\{ \beta_2 \left( ae + t^{\gamma} \right) + be - I(\tilde{f}) \right\}.
			\end{align}
		As before, since $\beta_2 \to -\infty$, $(ae + t^{\gamma})$ must be minimized. The minimizing value occurs along the lower boundary curve,
			\begin{equation*}
				g(e) = ae + \sum_{k = 1}^{\infty} r_{k}(e) \1{I_{k}}(e).
			\end{equation*}
		By the convexity of $l$, $-a > l^{\prime}(e_{n})$.
		Also by the convexity of $l$, for $e < e_{n}$,
			\begin{equation*}
				\frac{l(e_n) - r(e)^{\gamma}}{e_{n} - e} \le \frac{l(e_n) - l(e)}{e_{n} - e} < l^{\prime}(e_{n}) < -a.
			\end{equation*}
		Therefore $g(e) > g(e_{n})$ for $e < e_{n}$ implying that the minimizing values of $g$ must lie in the closed interval $I_{n+1}$. Let $e^{\ast}$ be a minimizing value of $g$. First assume that $\gamma \le \gamma_{n+1}$. Then by the analysis of Section \ref{lambertw}
			\begin{equation*}
				\partial_{+} g(e_{n}) = -s_{n+1} + 3\gamma \frac{n}{n+1} t_{n}^{\gamma-1} \ge 0.
			\end{equation*}
		Since $g$ transitions from convex to concave on $I_{n+1}$, the minimizing values of $g$ are in the set $\{e_{n},e_{n+1}\}$ and the first conclusion is realized depending on the value of $b$. Now assume that $\gamma > \gamma_{n+1}$. Then
			\begin{equation*}
				\partial_{+} g(e_{n}) = -s_{n+1} + 3\gamma \frac{n}{n+1} t_{n}^{\gamma-1} < 0,
			\end{equation*}
		which implies that $g$ has a minimum at $e^{\ast}$ where $e^{\ast} \in (e_{n},i_{n+1})$. 
	\end{proof}

This characterizes the limiting behavior of the model defined in Equation \ref{get} for $a = -s_{n}$ for some $n \in \N$ and any $\gamma > 1$. The case where $a$ is caught between consecutive critical directions is now examined. In other words, $-s_{n} > a > -s_{n+1}$ for some $n \in \N$ and $\gamma > 1$.

\begin{theorem} \label{between}
	Consider the generalized edge-triangle exponential random graph model defined in Equation \ref{get}. Let $\beta_1 = a\beta_2 + b$. Then Equation \ref{varproblem} holds; and, for $\gamma > 1$ and $-s_{n} > a > -s_{n+1}$, the set of maximizers $U \subset \mathcal{W}$ is 
		\begin{itemize}
			\item $U = \{ f \}$ if $\gamma > \tilde{\gamma}_{n}$ and $s_{n} < -a < \gamma \left( \frac{3n-3}{n+1} \right) t_{n}^{\gamma - 1}$ where $t(K_2,f) = e^{\ast}$ and $t(K_3,f) = r_{n}(e^{\ast})$ for $e_{n-1} < e^{\ast} < i_{n}$,

			\item $U = \{ f^{K_{n+1}} \}$ if $\gamma \le \frac{n+4}{6}$ or $\gamma \left( \frac{3n-1}{n+1} \right) t_{n}^{\gamma - 1} \le -a \le \gamma \left(\frac{3n}{n+1} \right) t_{n}^{\gamma - 1}$,

			\item $U = \{ f \}$ if $\gamma > \gamma_{n+1}$ and $\gamma \left(\frac{3n}{n+1} \right) t_{n}^{\gamma - 1} < -a < s_{n+1}$ where $t(K_2,f) = e^{\ast}$ and $t(K_3,f) = r_{n+1}(e^{\ast})$ for $e_{n} < e^{\ast} < i_{n+1}$;
		\end{itemize}
	where
		\begin{equation} \label{tildegamma}
			\tilde{\gamma}_{n} = -\frac{W_{0}\left(\tilde{q}(n)e^{\tilde{q}(n)} \right)}{\ln{\tilde{p}(n)}} - \frac{1}{\tilde{a}(n)}
		\end{equation}
	with
		\begin{equation}
			\tilde{p}(n) = \frac{(n-2)(n+1)^2}{n^3} \text{,\hspace{.2cm}} \tilde{a}(n) = -\frac{3}{n^2} \text{, and\hspace{.2cm}} \tilde{q}(n) = -\frac{\ln{\tilde{p}(n)}}{\tilde{a}(n)}.
		\end{equation}
\end{theorem}

	\begin{proof}
		Let $e^{\ast}$ be the minimizing value of $g$. According to the analysis that appears after Equation \ref{convex}, since $-s_{n} > a > -s_{n+1}$, $e^{\ast} \in (e_{n-1},i_{n}) \cup [e_{n}, i_{n+1})$. If $\gamma \le (n+4)/6$, then $g$ is concave down and so $e^{\ast} = e_{n}$. For $e < e_{n}$, 
			\begin{equation*}
				\frac{l(e_n) - r(e)^{\gamma}}{e_{n} - e} \le \frac{l(e_n) - l(e)}{e_n - e} < l^{\prime}(e_n).
			\end{equation*}
		This implies that if $a < - l^{\prime} (e_n)$, then $e^{\ast} \in [e_{n},i_{n+1})$. Similarly, if $a > - l^{\prime}(e_n)$, then $e^{\ast} \in (e_{n-1},i_{n})$ or $e^{\ast} = e_n$. First assume $a < -l^{\prime}(e_n)$. A necessary and sufficient condition for $e^{\ast} \in (e_{n},i_{n+1})$ is $\partial_{+} g(e_{n}) < 0$. This condition holds when 
			\begin{equation*}
				-a > 3\gamma\left(\frac{n}{n+1}\right)t_{n}^{\gamma-1};
			\end{equation*}
		and,
			\begin{equation*}
				s_{n+1} > 3\gamma\left(\frac{n}{n+1}\right)t_{n}^{\gamma-1}
			\end{equation*}
		if and only if $\gamma > \gamma_{n+1}$ by Section \ref{lambertw}. This implies that for $\gamma > \gamma_{n+1}$ and $s_{n+1} > -a > 3\gamma \left(\frac{n}{n+1}\right)t_{n}^{\gamma-1}$, $e^{\ast} \in (e_{n},i_{n+1})$. Contrarily, if
			\begin{equation*}
				3\gamma\left(\frac{n}{n+1}\right)t_{n}^{\gamma-1} \ge -a > l^{\prime}(e_{n}),
			\end{equation*}
		one finds that $\partial_{+} g(e_n) \ge 0$, thus $e^{\ast} = e_{n}$. Now suppose $a > -l^{\prime}(e_n)$. In this range of $a$ values, $e^{\ast} \in (e_{n-1},i_{n})$ or $e^{\ast} = e_{n}$. A sufficient condition for $e^{\ast} \in (e_{n-1}, i_{n})$ is $\partial_{-} g(e_{n}) \ge 0$. This condition is satisfied when 
			\begin{equation*}
				-a \le 3\gamma \left(\frac{n-1}{n+1}\right) t_{n}^{\gamma - 1};
			\end{equation*}
		and,
			\begin{equation*}
				s_{n} < 3\gamma \left(\frac{n-1}{n+1}\right) t_{n}^{\gamma - 1}
			\end{equation*}
		if and only if $\gamma > \tilde{\gamma}_{n}$ as defined in Equation \ref{tildegamma}. Therefore, if $\gamma > \tilde{\gamma}_{n}$ and $s_{n} < -a \le 3\gamma \left(\frac{n-1}{n+1}\right) t_{n}^{\gamma - 1}$, then $e^{\ast} \in (e_{n-1}, i_{n})$.

	\end{proof}

\begin{remark}
	An analysis similar to Lemma \ref{gammabound} may be done for the value of $\tilde{\gamma}_{n}$ and it is conjectured that $\tilde{\gamma}_{n} \sim \frac{4n}{9}$. The details will differ slightly due to the fact that $\tilde{\gamma}_{n}$ involves the branch $W_0$ of the Lambert $W$ function rather than $W_{-1}$.
\end{remark}

\begin{remark} \label{travel}
	This theorem indicates that for large enough $\gamma$, that is $\gamma \gtrsim \frac{4n}{9}$, the minimizing value of $g$, $e^{\ast}$, travels along the interval $(e_{n-1},e_{n+1})$. As $-a$ increases from $s_{n}$ to $s_{n+1}$, $e^{\ast}$ traverses this interval from left endpoint to right endpoint. Furthermore, as $\gamma$ increases, the inflection points $i_{n}$ on the interval $I_{n}$ get pushed arbitrarily close to the right endpoint $e_{n+1}$ and so $g$ becomes concave up on the entirety of the interval $I_{n}$ as $\gamma \to \infty$. This indicates a smooth transition between adjacent Tur\'an graphons as opposed to the case of $\gamma = 1$ where the transitions between Tur\'an graphons are abrupt jumps.
\end{remark}

The conclusions obtained in Sections \ref{<1} and \ref{>1} characterize the extremal asymptotic behavior of the generalized edge-triangle model through functional convergence in the cut topology in the space $\widetilde{\mathcal{W}}$. These can also be analyzed through a probabilistic lens. By combining Equation \ref{setconv} with the Theorems in the preceding sections and a diagonal argument, there exist subsequences of the form $\{(n_i, \beta_{2,i})\}_{i \in \N}$ where $n_{i} \to \infty$ and $\beta_{2,i} \to \infty$ or $-\infty$ as $i \to \infty$ such that the following holds. For $a$ and $b$ fixed, let $\{G_{i}\}_{i \in \N}$ be a sequence of random graphs drawn from the sequence of probability distributions $\PR_{n_{i}}^{(\beta_{2,i},\gamma)}$ where $G_{i}$ has $n_{i}$ vertices. Then 
	\begin{equation}
		\delta_{\square}\left( \tilde{f}^{G_{i}}, \widetilde{U} \right) \to 0 \hspace{.25cm} \text{in probability as} \hspace{.25cm} i \to \infty,
	\end{equation}
where $U \subset \mathcal{W}$ depends on $a$ and $b$ as in the preceding Theorems.

\subsection{Locating the Critical Point} \label{lcp}

The results of this section reveal the edge density of the limiting graphon for certain values of $\gamma$ as first the network size grows to infinity, then as the parameters diverge along straight lines. Fix $\gamma$, $a$, and $b$. The behavior for the case of $\gamma \le 1$ is fully characterized by Theorems \ref{yinrinfad}, \ref{gamma5/9} and Lemmas \ref{gammalessast}, \ref{gammaast}, and \ref{gammagreaterast}. Now for $\gamma > 1$, if $a > -s_2$ or $a \le -3\gamma$, Lemmas \ref{k2} and \ref{1} furnish the limiting graphon immediately. For $-s_2 \ge a > -3\gamma$, Proposition \ref{s2} and Theorems \ref{equal} and \ref{between} reveal the segment of $g$ on which the minimum must lie for $\gamma$ large enough. For smaller values of $\gamma$ these same results produce the exact limiting graphon as a certain Tur\'an graphon. For $\gamma$ sufficiently large, the exact limiting graphon is not clear as these results did not reveal the exact edge density, but rather an open interval in which the density must lie. Suppose that the minimum $e^{\ast}$ lies in the open interval between $e_{k-1}$ and $e_{k}$. Thus for $\gamma$ in this region, the problem of finding the limiting graphon reduces to finding the roots of the equation
	\begin{align} \label{zero}
		0 = a + &\frac{3(k-1)\gamma}{k(k+1)} \left( \frac{(k-1)(k-2\sqrt{k(k-e(k+1))})(k+\sqrt{k(k-e(k+1))})^2}{k^2 (k+1)^2} \right)^{\gamma - 1} \notag \\  &\cdot(k+\sqrt{k(k-e(k+1))}) 
	\end{align}
for $e \in I_{k}$. Using the substituion $x = k + \sqrt{k(k-e(k+1))}$ this is equivalent to finding the roots of
	\begin{equation}
		p(x;k,a,\gamma) = x\left( x^2 (3k-2x) \right)^{\gamma - 1} + \frac{ak(k+1)}{3\gamma(k-1)}\left( \frac{k^2 (k+1)^2}{k-1} \right)^{\gamma - 1}
	\end{equation}
for $x \in (k,k+1)$. Due to the structure of $p$, it is possible to write the root of $p$ in the interval $(k,k+1)$ as a nested radical. Thus for $\gamma$ large enough, the exact limiting edge density is determined. Let $c$ be defined as
	\begin{equation}
		c = -\frac{ak(k+1)}{3\gamma(k-1)}\left( \frac{k^2 (k+1)^2}{k-1} \right)^{\gamma - 1}.
	\end{equation}
Solving Equation \ref{zero} simplifies to solving
	\begin{equation}
		x^{2\gamma - 1}(3k-2x)^{\gamma - 1} - c = 0
	\end{equation}
in the interval $(k,k+1)$. Rearranging and letting $m = \gamma - 1$ and $n = 1/(1-2\gamma)$, one finds the identity
	\begin{equation} \label{polyid}
		x = \frac{3k}{2} - \frac{\sqrt[m]{c}}{2} \sqrt[mn]{x} .
	\end{equation}
By applying the Identity \ref{polyid} infinitely many times one obtains the nested radical representation of the root. This nested radical converges to $x^{\ast} > k$ by choosing a sufficient starting value and examining the sequence of partial iterates using the monotone convergence theorem; moreover, $x^{\ast}$ is the root of $p$ in $(k,k+1)$. Thus $x^{\ast}$ is the root that solves the variational problem when converted back to the variable $e$. This realization leads to the following Corollary of Proposition \ref{s2}, Theorem \ref{equal}, and Theorem \ref{between}. Recall the definition of $\tilde{\gamma}_{k}$ from Equation \ref{tildegamma}.
	\begin{corollary}
		Consider the generalized edge-triangle model defined in Equation \ref{get}. Suppose $\gamma > 1$ and $-s_{2} \ge a > -3\gamma$. Then the results of Proposition \ref{s2}, Theorem \ref{equal}, and Theorem \ref{between} hold. Furthermore, if $e^{\ast}$ lies in the open interval $(e_{k-1},e_{k})$ as obtained in these results and $\gamma > \max\{1, \tilde{\gamma}_{k} \}$, then
			\begin{equation}
				e^{\ast} = \frac{k^2 - \left( x^{\ast} - k \right)^2}{k(k+1)},
			\end{equation}
		with $m,n,$ and $c$ defined as above and
			\begin{equation}
				x^{\ast} = \lim_{n \to \infty} x_{n}
			\end{equation}
		where $x_1 = k+1$ and 
			\[ x_{i+1} = \frac{3k}{2} - \frac{\sqrt[m]{c}}{2} \sqrt[mn]{x_{i}}. \]
	\end{corollary}


	\begin{proof}
		Let $\gamma > 1$. Suppose that $e^{\ast}$ lies in the open interval $(e_{k-1},e_{k})$. In Theorem \ref{equal}, this occurs for $-a = s_{k}$ and $\gamma > \gamma_{k}$. In Theorem \ref{between}, this occurs for 
			\[ s_{k} < -a < 3\gamma \left(\frac{k-1}{k+1}\right) t_{k}^{\gamma-1} \]
		and $\gamma > \tilde{\gamma}_{k}$. Define the sequence $\{x_{i}\}$ such that $x_1 = k+1$ and
			\[ x_{i+1} = \frac{3k}{2} - \frac{\sqrt[m]{c}}{2} \sqrt[mn]{x_{i}}  \]
		for all $i \in \N$. This sequence is decreasing. Now, if $\gamma \ge \tilde{\gamma}_{k}$, then 
			\[ s_{k} \le 3\gamma \left(\frac{k-1}{k+1}\right) t_{k}^{\gamma-1}. \]
		It follows since $-a \le 3\gamma \left(\frac{k-1}{k+1}\right) t_{k}^{\gamma-1}$ that $c \le k^{3\gamma - 2}$. Clearly $x_1 > k$. Suppose that $x_{i-1} > k$, then
			\begin{align*}
				x_{i} &>  \frac{3k}{2} - \frac{\sqrt[m]{c}}{2} \sqrt[mn]{k} \\
					&= \frac{3k}{2} - \frac{1}{2} c^{\frac{1}{\gamma-1}} \left( k \right)^{\frac{1-2\gamma}{\gamma-1}} \\
					&\ge \frac{3k}{2} - \frac{1}{2} \left(k^{3\gamma-2}\right)^{\frac{1}{\gamma-1}} \left( k \right)^{\frac{1-2\gamma}{\gamma-1}} \\
					&= k.
			\end{align*}
		Thus $x_{i} > k$ for all $i \in \N$. By the monotone convergence theorem, $x_{i}$ converges to some $x^{\ast}$ in $[k,k+1)$. In fact, since $e^{\ast}$ is known to be in the interval $(e_{k-1},e_{k})$, $x^{\ast} \in (k,k+1)$ because $e=e_{k}$ corresponds to $x=k$.
	\end{proof}

\begin{table}[h!]
\begin{equation*}
\begin{array}{|c||ccccc|}
\hline
a  & -s_2 & -s_2 & -s_2 & -s_2 & -1.1 \\
\hline
\gamma & 2 & 4 & 10 & 100 & 2 \\
\hline
e^{\ast} & 0.575 & 0.599 & 0.625 & 0.658 & 0.703 \\ 
\hline
\end{array}
\end{equation*}
\caption{Limiting graphon edge density $e^{\ast}$ for given parameter values.} \label{table}
\end{table}

\begin{remark}
Some numerical results for finding the root of interest of the function $p$ are now presented in simple cases for the purposes of illustration. As an example, examine $\gamma = 2$ and $a = -s_2$. By Proposition \ref{s2}, the minimum of $g$ lies on the second segment, so $k=2$. This now reduces to the problem of determing the roots of the polynomial 
	\begin{equation*}
		p(x;2,-s_2,2) = -2x^4 + 6x^3 - \frac{32}{3}
	\end{equation*}
in the interval $(2,3)$. The examples in Table \ref{table} considering $\gamma = 4,10,$ and $100$ are handled similarly. As another example, let $\gamma = 2$ and $a = -1.1$. By Theorem \ref{between}, $2 > \tilde{\gamma}_3$ and $s_3 < -a < 9/8$ so by bullet $(1)$, $e^{\ast}$ lies on the segment $k = 3$ and we obtain the polynomial
	\begin{equation*}
		p(x;3,-1.1,2) = -2x^4 + 9x^3 - \frac{396}{5}.
	\end{equation*}
The corresponding $e^{\ast}$ for these cases are contained in Table \ref{table}.
\end{remark}

\section{Generalized Edge-Clique Model}

This section concerns a more general model by replacing the triangle homomorphism density with that of any clique on $s$ vertices. Let $s \ge 3$ and define $g_{s}(e)$ to be the minimum achievable density of the complete graph on $s$ vertices, $K_{s}$, that can appear as a subgraph in any graph on $n$ vertices with edge density $e$. Lov\'asz and Simonovits conjectured in \cite{LS} that the asymptotic lower bound on the density of $K_{s}$ as a subgraph of $G_{n}$ for a fixed edge density $e$ is given by
	\begin{align} \label{grho}
		g_{s}(e) = &\frac{(t-1)!}{(t-s+1)!(t(t+1))^{s-1}} \left( t - (s-1) \sqrt{t(t-e(t+1))} \right) \notag \\
			&\cdot\left( t + \sqrt{t(t-e(t+1)} \right)^{s-1},
	\end{align}
where $e \in \left[ \frac{t-1}{t}, \frac{t}{t+1} \right]$ and $e > 1 - \frac{1}{s-1}$. Furthermore, they conjectured that this bound is tight. It was proven by Razborov in \cite{R} that the lower bound on the asymptotic density is tight when $s = 3$. This bound proved indispensible in determining the solution to the variational problem in the $\beta_2 \to -\infty$ limit. The next step was acheived by Nikiforov in 2011 who was able to verify the case of $s = 4$ in \cite{N}. Reiher settled the conjecture when he proved in 2016 that this bound is tight and holds for all $s \ge 2$ \cite{Re}. By the Kruskal-Katona Theorem, the upper bound on the realizable density of $K_s$ is $G_{s}(e) = e^{\frac{s}{2}}$. Thus a similar analysis to the one exhibited in this paper can be applied to the exponential random graph model with Hamiltonian consisting of edge density and the homomorphism density of a clique on $s$ vertices for $s \ge 3$,
	\begin{equation} \label{2dim}
		\PR_{n}^{(\beta,\gamma)}(G_{n}) = \exp\left( n^2 \left( \beta_1 t(K_{2}, G_{n}) + \beta_2 t(K_{s}, G_{n})^{\gamma} - \psi_{n}^{(\beta,\gamma)} \right) \right),
	\end{equation}
where $\beta \in \R^{2}$ and $\gamma \in \R_{\ge 0}$.

This paper dealt solely with the case of $s = 3$. The next Theorem regards the positive $\beta_2$ limit of the generalized edge-clique model for any clique $K_s$ with $s \ge 3$. 

\begin{theorem} \label{rclique}
	Consider the generalized edge $s$-clique exponential random graph model defined in Equation \ref{2dim} for $s \ge 3$. Let $\beta_1 = a\beta_2 + b$. Then
		\begin{equation}
			\lim_{\beta_2 \to \infty} \sup_{\tilde{f} \in \widetilde{F}^{\ast}(\beta_2)} \left\{ \delta_{\square} \left( \tilde{f}, \widetilde{U} \right) \right\} = 0, 
		\end{equation}
	where for $\gamma \ge 2/s$, the set $U \subset \mathcal{W}$ is:
		\begin{itemize}
			\item $U = \left\{ 0 \right\}$ if $a < -1$ or $a = -1$ and $b < 0$,
			\item $U = \left\{ 0,1 \right\}$ if $a = -1$ and $b = 0$,
			\item $U = \left\{ 1 \right\}$ if $a > -1$ or $a = -1$ and $b > 0$;
		\end{itemize}
	and for $\gamma < 2/s$, the set $U \subset \mathcal{W}$ is:
		\begin{itemize}
			\item $U = \left\{ 1 \right\}$ if $a \ge -s\gamma / 2$, and
			\item $U = \left\{ f \right\}$ if $a < -s\gamma / 2$,
		\end{itemize}
	where 
		\begin{equation}
			f(x,y) = \left\{  \begin{array}{ll} 1 & 0 \le x,y \le \left( \frac{-2a}{s\gamma} \right)^{\frac{1}{s\gamma - 2}} \\ 0 & \text{otherwise.} \end{array} \right.
		\end{equation}
\end{theorem}

\begin{proof}
	Let $s \ge 3$. The limiting normalization constant takes the form
		\begin{align*}
			\psi_{\infty}^{(\beta,\gamma)} = \sup_{\tilde{f} \in \widetilde{\mathcal{W}}} \left\{ \beta_2 (at(K_2,\tilde{f}) + t(K_s,\tilde{f})^{\gamma}) + bt(K_2,\tilde{f}) - I(\tilde{f}) \right\}.
		\end{align*}
	As $\beta_2 \to \infty$, the term $bt(K_2,\tilde{f}) - I(\tilde{f})$ is bounded. Therefore the function
		\begin{equation*}
			g(\tilde{f}) = at(K_2,\tilde{f}) + t(K_s,\tilde{f})^{\gamma}
		\end{equation*}
	must be maximized over the graphon space. Let $e = t(K_2,\tilde{f})$. Using the Kruskal-Katona bound, this is equivalent to maximizing 
		\[ g(e) = ae + e^{s\gamma/2} \]
	on the interval $[0,1]$. If $\gamma \ge 2/s$, then $g^{\prime \prime}(e) \ge 0$. Thus the maximizer is either the empty or complete graphon and the conclusions for $\gamma \ge 2/s$ are gathered by considering the values of $a$ and $b$. If $\gamma < 2/s$ then there are two cases. For $a \ge -s\gamma / 2$, $g^{\prime}(e) > 0$ for $e > 0$ and so the maximizer is the complete graphon. If $a < -s\gamma / 2$, then $g$ is first increasing then decreasing. Therefore the maximizer lies in $(0,1)$ and is determined by solving $g^{\prime}(e)=0$.
\end{proof}

\section{Acknowledgments}
The author would like to express many thanks to his advisor, Dr.\ Mei Yin, for introducing him to this problem, as well as all of the fruitful conversations and indispensable guidance. Also, a great many thanks are owed to the author's fianc\'e for several proofreads and bearing with him every step of the way.

\section{Appendix}

In this section, the statements and proofs for the case of $5/9 < \gamma \le \log_{\frac{27}{16}}(3/2)$. The variational solution relies heavily on the behavior of the sequence $\{s_{k}\}$, which transitions between increasing and decreasing on this interval of $\gamma$ values. Recall that there is a critical value of $\gamma$, denoted $\gamma^{\ast}$ such that when $\gamma = \gamma^{\ast}$, $s_2 = 3\gamma$ and 
	\begin{equation} \label{critgam}
		\gamma^{\ast} = \frac{W_0 \left(2\ln\left(9/2\right)\right)}{\ln\left(9/2\right)}.
	\end{equation}

\begin{figure}[t!]
	\centering
		\begin{subfigure}
			\centering
			\includegraphics[keepaspectratio, width=.49\textwidth]{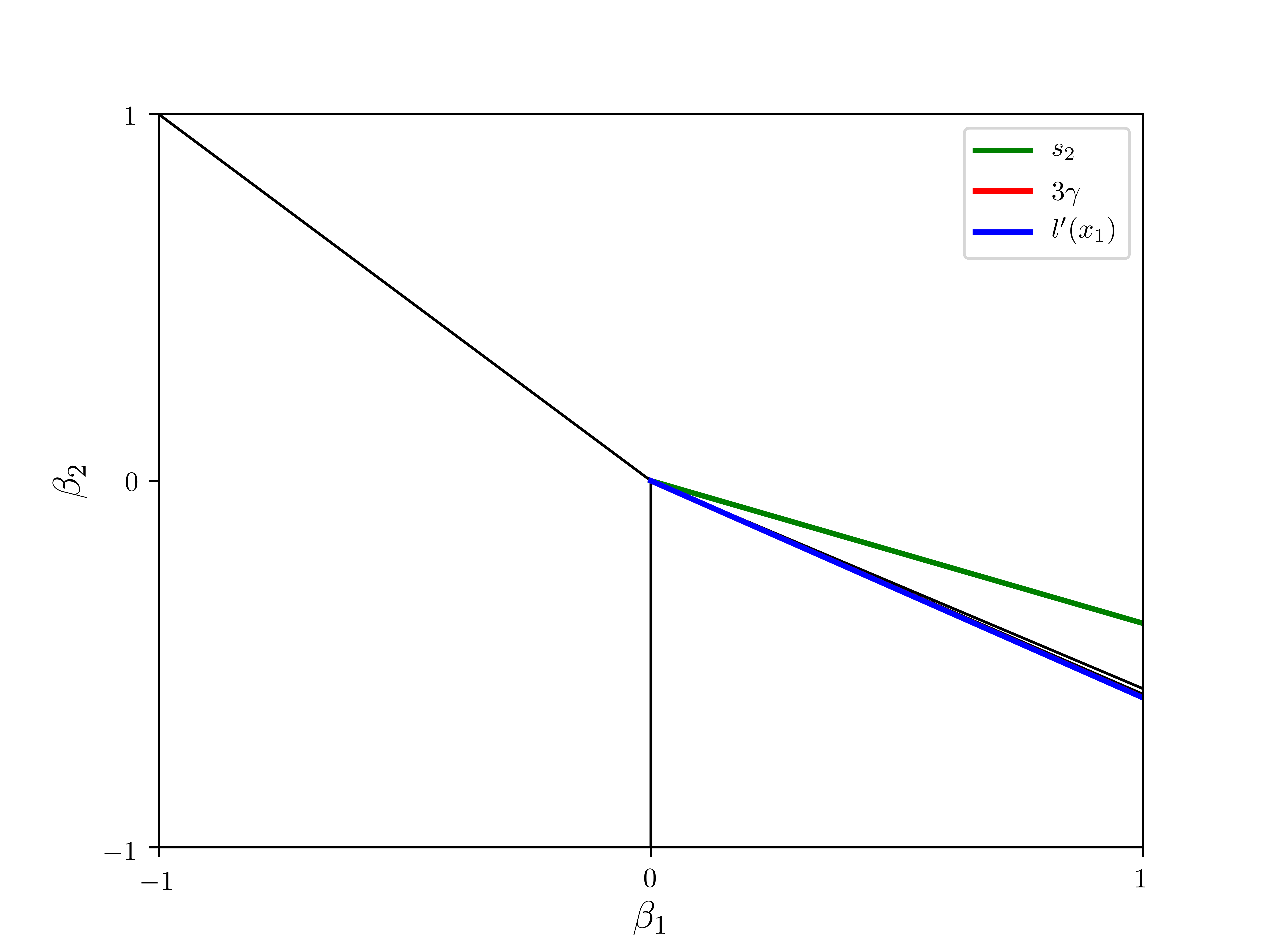}
			\label{fig1:sub1}
		\end{subfigure}
		\begin{subfigure}
			\centering
			\includegraphics[keepaspectratio, width=.49\textwidth]{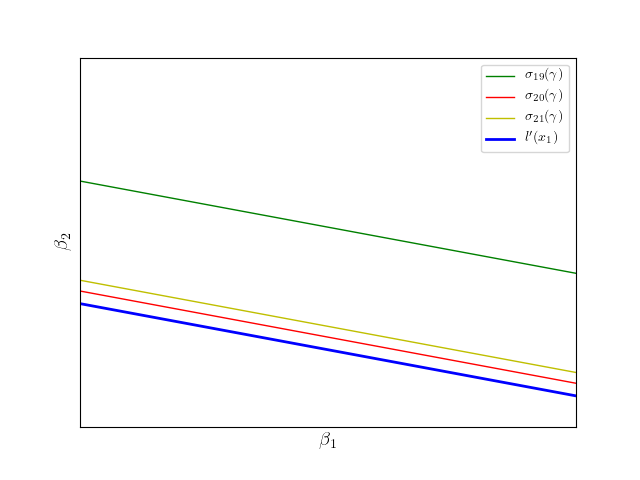}
			\label{fig1:sub2}
		\end{subfigure}
	\caption{The plot at left displays the vectors of the critical directions for $\gamma < \gamma^{\ast}$. Here the vector associated with the limit $3\gamma$ is below the vector associated with $s_2$. The plot at right shows that the value of $n$ from Lemma \ref{gammalessast} is $n = 20$ for $l^{\prime}(x_2) < -a \le s_{20}$. }
\end{figure}

\begin{lemma} \label{gammalessast}
	Consider the generalized edge-triangle exponential random graph model defined in Equation \ref{get}. Let $\beta_1 = a\beta_2 + b$. Then
		\begin{equation} \label{sup2}
			\lim_{\beta_2 \to -\infty} \sup_{\tilde{f} \in \widetilde{F}^{\ast}(\beta_2)} \left\{ \delta_{\square} \left( \tilde{f}, \widetilde{U} \right) \right\} = 0.
		\end{equation}
	Assume that $5/9 < \gamma < \gamma^{\ast}$ is defined as in Equation \ref{critgam}. Let $\gamma_{n}^{\ast}$ be defined as
		\begin{equation}
			\gamma_{n}^{\ast} = \frac{\ln\left( \frac{(n-2)(n+1)}{n(n-1)} \right)}{\ln\left( \frac{(n-2)(n+1)^2}{n^3} \right)}.
		\end{equation}
	Recall from Lemma \ref{sk} 
		\begin{equation}
			x_2 = \frac{1}{4} \left( 1 + \frac{1}{\sqrt{2\gamma - 1}} \right).
		\end{equation}
	Then $\gamma_{n}^{\ast} \to 2/3$ as $n \to \infty$ and the set $U$ is defined as follows:
		\begin{itemize}
			\item If $-a \le l^{\prime}(x_2)$, then $U = \{f^{K_2}\}$.
			\item If $-a \ge s_{2}$, then $U = \{1\}$.
			\item If $3\gamma \le -a < s_{2}$, then the following cases hold.
				\begin{itemize}
					\item If $a > -2$ or $a = -2$ and $b < 0$, then $U = \{ f^{K_{2}} \}$.
					\item If $a = -2$ and $b = 0$, then $U = \{ f^{K_2}, 1\}$.
					\item If $a < -2$ or $a = -2$ and $b > 0$, then $U = \{ 1 \}$.
				\end{itemize}
			\item Suppose $l^{\prime}(x_2) < -a < 3\gamma$. Let $n$ be the least positive integer such that $-a \le s_{n}$ and $p_{n} > x_{1}$ where $l^{\prime}(p_{n}) = s_{n}$, thus there is also a corresponding $\gamma_{n}^{\ast}$. Now define $a_{n} = \frac{2(n+1)}{n-1} t_{n}^{\gamma}$. Firstly, if $\gamma > \gamma_{n}^{\ast}$, then the following cases hold
				\begin{itemize}
					\item If $-a < a_{n-1}$ or $-a = a_{n-1}$ and $b < 0$, then $U = \{ f^{K_{2}} \}$.
					\item If $-a = a_{n-1}$ and $b = 0$, then $U = \{ f^{K_2}, f^{K_{n}} \}$.
					\item If $-a = s_{n}$ and $b < 0$ or $a_{n-1} < -a < s_{n}$ or $-a = a_{n-1}$ and $b > 0$, then $U = \{ f^{K_{n}} \}$.
					\item If $-a = s_{n}$ and $b = 0$, then $U = \{ f^{K_{n}}, f^{K_{n+1}} \}$.
					\item If $-a = s_{n}$ and $b > 0$, then $U = \{ f^{K_{n+1}} \}$.
				\end{itemize}
			Secondly, if $\gamma < \gamma_{n}^{\ast}$, then $U = \{ f^{K_2} \}$. Lastly, consider $\gamma = \gamma_{n}^{\ast}$ giving the following scenarios.
				\begin{itemize}
					\item If $-a < s_{n}$ or $-a = s_{n}$ and $b < 0$, then $U = \{ f^{K_2} \}$.
					\item If $-a = s_{n}$ and $b = 0$, then $U = \{ f^{K_2}, f^{K_{n}}, f^{K_{n+1}}\}$.
					\item If $-a = s_{n}$ and $b > 0$, then $U = \{ f^{K_{n+1}} \}$.
				\end{itemize}
		\end{itemize}
\end{lemma}

	\begin{proof}
		Let $\beta_1 = a\beta_2 + b$ and $\frac{5}{9} < \gamma < \gamma^{\ast}$. Then 
			\begin{align}
				\psi_{\infty}^{(\beta,\gamma)}(T_{\beta}) &= \sup_{\tilde{f} \in \widetilde{W}} \left\{ \beta_1 t(K_2,\tilde{f}) + \beta_2 t(K_3,\tilde{f})^{\gamma} - I(\tilde{f}) \right\} \notag \\
				&= \sup_{\tilde{f} \in \widetilde{W}} \left\{ \beta_2 \left( ae + t^{\gamma} \right) + be - I(\tilde{f}) \right\}.
			\end{align}
		The preceding maximization problem must be solved. As $\beta_2 \to -\infty$, $(be - I)$ is bounded, so $ae + t^{\gamma}$ must be minimized. Minimizing this expression occurs along the lower boundary curves of the region of realizable densities. Let $g(e) = ae + r(e)$ where $r(e)$ is defined as in Equation \ref{LowerBoundaryCurve}. The assumption that $\gamma < \gamma^{\ast}$ implies that $3\gamma < s_{2}$. Since the sequence $\{s_{k}\}$ is first decreasing then increasing, $s_{2} > s_{n}$ for all $n > 2$. First assume that $-a \ge s_{2}$. Then $-a > s_{n}$ for all $n > 2$. Note that if $a < -s_{n}$, then $g(e_{n}) < g(e_{n-1})$. Thus it holds that $g(e_{n}) < g(e_{n-1})$ for all $n > 2$ with possible equality at $n = 2$. Regardless of the equality at $n = 2$, the minimizing value of $g$ occurs at $1$ which translates to the complete graphon and so $U = \{1\}$.

		Now assume that $-a \le l^{\prime}(x_2)$. Then $-a < s_{n}$ for all $n \in \N$ with possible equality at some point $m \in \N$. Note that for all $\gamma$, $s_{2} > l^{\prime}(x_2)$. This ensures that $m > 2$. For all $n \in \N$, this implies that $g(e_{n-1}) < g(e_{n})$ with possible equality at $m > 2$. Thus the minimum of $g$ occurs at $e_1 = \frac{1}{2}$ and so $U = \{ f^{K_{2}} \}$.

		Suppose that $3\gamma \le -a < s_{2}$. Let $m = \sup\{ n \in \N : -a \le s_{n} \}$. This supremum must exist and be finite since $\{s_{k}\}$ is first decreasing then increasing to $3\gamma$. Then $-a < s_{n}$ for $n < m$, $-a \le s_{m}$, and $-a > s_{n}$ for $n > m$. These values for $a$ correspond to properties of the sequence $g(e_{n})$. Thus $g(e_{n-1}) < g(e_{n})$ for $n < m$, $g(e_{m-1}) \le g(e_{m})$, and $g(e_{n-1}) > g(e_{n})$ for $n > m$. The possible minimizing values for the function $g$ are $\frac{1}{2}$ and $1$. This requires relating $g(e_1)$ to $g(1)$. If $a < -2$ or $a = -2$ and $b > 0$, then $U = \{ 1 \}$. If $a > -2$ or $a = -2$ and $b < 0$, then $U = \{ f^{K_2} \}$. If $a = -2$ and $b = 0$, then $U = \{f^{K_2}, 1\}$.

		Lastly, suppose that $l^{\prime}(x_2) < -a < 3\gamma$. Let $n_1$ and $n_2$ be defined as
			\[ n_1 = \sup \{ n \in \N : -a \le s_{n} \hspace{.2cm}\&\hspace{.2cm} p_{n} < x_2 \text{  where  } l^{\prime}(p_{n}) = s_{n} \} \]
		and
			\[ n_2 = \inf \{ n \in \N : -a \le s_{n} \hspace{.2cm}\&\hspace{.2cm} p_{n} > x_2 \text{  where  } l^{\prime}(p_{n}) = s_{n} \}.\]
		Then $-a \le s_{n}$ for $n \le n_1$ and $n \ge n_2$ while $-a > s_{n}$ for $n_1 < n < n_2$. This gives that $g(e_{n-1}) < g(e_{n})$ for $n < n_1$ and $n > n_2$ with possible equality at $n = n_1$ or $n = n_2$. Also $g(e_{n-1}) > g(e_{n})$ for $n_1 < n < n_2$. This means that the possible minimizers of $g$ occur at $e_1$, $e_{n_2 - 1}$, and $e_{n_2}$ is included if $-a = s_{n_2}$. First assume that $-a < s_{n_2}$. Then the possible minimizers of $g$ are $e_1$ and $e_{n_2 - 1}$. Now the relation between $g(e_1)$ and $g(e_{n_2 - 1})$ must be determined:
			\begin{align*}
				g(e_1) > g(e_{n_{2}-1}) &\iff \frac{a}{2} > ae_{n_{2}-1} + t_{n_{2}-1}^{\gamma} \iff a\left(\frac{1}{2} - e_{n_{2}-1}\right) > t_{n_{2}-1}^{\gamma} \\
					&\iff a < \frac{t_{n_{2}-1}^{\gamma}}{\frac{1}{2} - e_{n_{2}-1}} = -2\frac{n_{2}}{(n_{2}-2)} t_{n_{2}-1}^{\gamma}.
			\end{align*}
		Define $a_{n} = \frac{2(n+1)}{(n-1)} t_{n}^{\gamma}$. Therefore $g(e_1) > g(e_{n_{2}-1})$ if $-a > a_{n_{2}-1}$, $g(e_1) < g(e_{n_{2}-1})$ if $-a < a_{n_{2}-1}$, and equal if $-a = a_{n_{2}-1}$. Now $a_{n_{2}-1} > s_{n_{2}}$ if the subsequent string of equivalences hold
			\begin{align*}
				a_{n_{2}-1} > s_{n_{2}} &\iff \frac{2n_{2}}{(n_{2}-2)} t_{n_{2}-1}^{\gamma} > \frac{t_{n_{2}}^{\gamma} - t_{n_{2}-1}^{\gamma}}{e_{n_{2}} - e_{n_{2}-1}} \\ 
					&\iff \frac{2}{(n_{2}+1)(n_{2}-2)} t_{n_{2}-1}^{\gamma} > t_{n_{2}}^{\gamma} - t_{n_{2}-1}^{\gamma} \\ 
					&\iff \frac{2}{(n_{2}+1)(n_{2}-2)} > \left( \frac{n_{2}^3}{(n_{2}-2)(n_{2}+1)^2} \right)^{\gamma} - 1 \\
					&\iff \ln\left( \frac{n_{2}(n_{2}-1)}{(n_{2}+1)(n_{2}-2)} \right) > \gamma \ln\left( \frac{n_{2}^3}{(n_{2}-2)(n_{2}+1)^2} \right) \\
					&\iff \gamma < \frac{\ln\left( \frac{n_{2}(n_{2}-1)}{(n_{2}+1)(n_{2}-2)} \right)}{\ln\left( \frac{n_{2}^3}{(n_{2}-2)(n_{2}+1)^2} \right)} = \gamma_{n_{2}}^{\ast}.
			\end{align*}
		If $\gamma < \gamma_{n}^{\ast}$, then $a_{n-1} > s_{n}$. Similarly, if $\gamma > \gamma_{n}^{\ast}$, then $a_{n-1} < s_{n}$, and, if $\gamma = \gamma_{n}^{\ast}$, then $a_{n-1} = s_{n}$. Consider the case where $\gamma < \gamma_{n_{2}}^{\ast}$. Since $-a < s_{n_{2}}$, $-a < a_{n_{2}-1}$ as well. Therefore $g(e_{1}) < g(e_{n_2 - 1})$ and so $U = \{ f^{K_{2}} \}$. If $\gamma = \gamma_{n_2}^{\ast}$, then $a_{n_{2}-1} = s_{n_2}$ and so $-a < a_{n_{2}-1}$, further implying that $U = \{ f^{K_{2}} \}$. Suppose that $\gamma > \gamma_{n_{2}}^{\ast}$. This yields $a_{n_2 - 1} < s_{n_{2}}$. If $-a < a_{n_{2}-1}$, then $g(e_{1}) < g(e_{n_2 - 1})$ giving $U = \{ f^{K_2} \}$. If $-a = a_{n_{2} - 1}$, then $g(e_{1}) = g(e_{n_{2} - 1})$. Therefore the value of the parameter $b$ will play a role in determing the maximizing graphon. If $b = 0$, then $U = \{ f^{K_2}, f^{K_{n_2}} \}$. If $b > 0$, then $U = \{ f^{K_{n_2}} \}$. If $b < 0$, then $U = \{ f^{K_{2}} \}$. Now suppose that $a_{n_2 - 1} < -a < s_{n_2}$. Then $g(e_{1}) > g(e_{n_2 - 1})$ and $U = \{ f^{K_{n_2}} \}$.

		Lastly, assume that $-a = s_{n_2}$. The composition of the set $U$ succeeds from a similar analysis as in the preceding paragraph. 
	\end{proof}

\begin{lemma} \label{gammaast}
	Consider the generalized edge-triangle exponential random graph model defined in Equation \ref{get}. Let $\beta_1 = a\beta_2 + b$. Then
		\begin{equation} \label{sup1}
			\lim_{\beta_2 \to -\infty} \sup_{\tilde{f} \in \widetilde{F}^{\ast}(\beta_2)} \left\{ \delta_{\square} \left( \tilde{f}, \widetilde{U} \right) \right\} = 0.
		\end{equation}
	Assume that $\gamma = \gamma^{\ast}$ is defined as above in Equation \ref{critgam}. Also assume that $\gamma_{n}^{\ast}$ and $x_{1}$ are defined as in Lemma \ref{gammalessast}. Then the set $U$ is defined as follows:
		\begin{itemize}
			\item If $-a \le l^{\prime}(x_2)$, then $U = \{f^{K_2}\}$.
			\item If $-a \ge s_{2}$, then $U = \{1\}$.
			\item Suppose $l^{\prime}(x_2) < -a < s_{2}$. Let $n$ be the least positive integer such that $-a \le s_{n}$ and $p_{n} > x_{1}$ where $l^{\prime}(p_{n}) = s_{n}$, thus there is also a corresponding $\gamma_{n}^{\ast}$. Now define $a_{n} = \frac{2(n+1)}{n-1} t_{n}^{\gamma}$. Firstly, if $\gamma^{\ast} > \gamma_{n}^{\ast}$, then the following cases hold.
				\begin{itemize}
					\item If $-a = a_{n-1}$ and $b < 0$ or $-a < a_{n-1}$, then $U = \{ f^{K_{2}} \}$.
					\item If $-a = a_{n-1}$ and $b = 0$, then $U = \{ f^{K_2}, f^{K_{n}} \}$.
					\item If $-a = s_{n}$ and $b < 0$ or $-a = a_{n-1}$ and $b > 0$ or $a_{n} < -a < s_{n}$, then $U = \{ f^{K_{n}} \}$.
					\item If $-a = s_{n}$ and $b = 0$, then $U = \{ f^{K_{n}}, f^{K_{n+1}} \}$.
					\item If $-a = s_{n}$ and $b > 0$, then $U = \{ f^{K_{n+1}} \}$.
				\end{itemize}
			Lastly, if $\gamma^{\ast} < \gamma_{n}^{\ast}$, then $U = \{ f^{K_2} \}$.
		\end{itemize}
\end{lemma}

\begin{proof}
	The proof of this lemma follows in a similar fashion to the proof of Lemma \ref{gammalessast} and is therefore omitted. 
\end{proof}

\begin{lemma} \label{gammagreaterast}
	Consider the generalized edge-triangle exponential random graph model defined in Equation \ref{get}. Let $\beta_1 = a\beta_2 + b$. Then
		\begin{equation} \label{sup3}
			\lim_{\beta_2 \to -\infty} \sup_{\tilde{f} \in \widetilde{F}^{\ast}(\beta_2)} \left\{ \delta_{\square} \left( \tilde{f}, \widetilde{U} \right) \right\} = 0.
		\end{equation}
	Assume that $\gamma^{\ast} < \gamma < \log_{\frac{27}{16}}(3/2)$, where $\gamma^{\ast}$ is defined in Equation \ref{critgam}. Also assume that $\gamma_{n}^{\ast}$ and $x_2$ are defined as in Lemma \ref{gammalessast}. Then the set $U$ is defined as follows:
		\begin{itemize}
			\item If $-a \le l^{\prime}(x_2)$, then $U = \{f^{K_2}\}$.
			\item If $-a \ge 3\gamma$, then $U = \{1\}$.
			\item If $s_{2} < -a < 3\gamma$, then the following cases are possible. 
				\begin{itemize}
					\item If $a > -2$ or $a = -2$ and $b < 0$, then $U = \{ f^{K_{2}} \}$.
					\item If $a = -2$ and $b = 0$, then $U = \{ f^{K_2}, 1\}$.
					\item If $a < -2$ or $a = -2$ and $b > 0$, then $U = \{ 1 \}$.
				\end{itemize}
			\item Suppose $l^{\prime}(x_2) < -a < s_{2}$. Let $n$ be the least positive integer such that $-a \le s_{n}$ and $p_{n} > x_{1}$ where $l^{\prime}(p_{n}) = s_{n}$, thus there is also a corresponding $\gamma_{n}^{\ast}$. Now define $a_{n} = \frac{2(n+1)}{n-1} t_{n}^{\gamma}$. Firstly, if $\gamma > \gamma_{n}^{\ast}$, then the following cases are possible.
				\begin{itemize}
					\item If $-a < a_{n-1}$ or $-a = a_{n-1}$ and $b < 0$, then $U = \{ f^{K_{2}} \}$.
					\item If $-a = a_{n-1}$ and $b = 0$, then $U = \{ f^{K_2}, f^{K_{n}} \}$.
					\item If $-a = s_{n}$ and $b < 0$ or $a_{n-1} < -a < s_{n}$ or $-a = a_{n-1}$ and $b > 0$, then $U = \{ f^{K_{n}} \}$.
					\item If $-a = s_{n}$ and $b = 0$, then $U = \{ f^{K_{n}}, f^{K_{n+1}} \}$.
					\item If $-a = s_{n}$ and $b > 0$, then $U = \{ f^{K_{n+1}} \}$.
				\end{itemize}
			Secondly, if $\gamma < \gamma_{n}^{\ast}$, then $U = \{ f^{K_2} \}$. Lastly, consider $\gamma = \gamma_{n}^{\ast}$ giving the following scenarios.
				\begin{itemize}
					\item If $-a < s_{n}$ or $-a = s_{n}$ and $b < 0$, then $U = \{ f^{K_2} \}$.
					\item If $-a = s_{n}$ and $b = 0$, then $U = \{ f^{K_2}, f^{K_{n}}, f^{K_{n+1}}\}$.
					\item If $-a = s_{n}$ and $b > 0$, then $U = \{ f^{K_{n+1}} \}$.
				\end{itemize}
		\end{itemize}
\end{lemma}

\begin{proof}
	The proof of this lemma follows in a similar fashion to the proof of Lemma \ref{gammalessast} and is therefore omitted. 
\end{proof}

\end{document}